%% file: ex_article.tex
\documentclass[hidelinks,onefignum,onetabnum]{siamart251216}

\input{ex_shared}

\makeatletter
\newcommand{\IndentedIf}{%
  \algpx@endCodeCommand%
  \addtolength{\ALG@tlm}{\algorithmicindent}%
  \ALG@thistlm\ALG@tlm%
  \ALG@makeentity{If}}
\newcommand{\EndWhileAfterIndent}{%
  \algpx@endCodeCommand%
  \addtolength{\ALG@tlm}{-\algorithmicindent}%
  \ALG@thistlm\ALG@tlm%
  \ALG@makeentity{EndWhile}}

\newcommand{\AlgorithmEqTagShift}{%
  \def\tagform@##1{%
    \maketag@@@{(\ignorespaces##1\unskip\@@italiccorr)}%
  }%
}

\newcommand{\AlgorithmEquation}[2]{%
  \algpx@endCodeCommand%
  \Statex
  \hspace*{\ALG@tlm}%
  \begin{minipage}{\dimexpr\linewidth-\ALG@tlm\relax}%
    \AlgorithmEqTagShift
    \setlength{\abovedisplayskip}{2pt}%
    \setlength{\belowdisplayskip}{2pt}%
    \setlength{\abovedisplayshortskip}{2pt}%
    \setlength{\belowdisplayshortskip}{2pt}%
    \begin{equation}\label{#1}
      #2
    \end{equation}%
  \end{minipage}%
  \vspace{-4.2pt}%
}
\makeatother

\ifpdf
\hypersetup{
  pdftitle={A Safeguarded Accelerated Proximal ADMM Algorithm for Solving Structured Cardinality Penalized Optimization Problems},
  pdfauthor={Wei Bian}
}
\fi

\newsiamremark{assumption}{Assumption}
\crefname{assumption}{Assumption}{Assumptions}
\usepackage{booktabs}
\usepackage{listings}
\usepackage{multirow}

\begin{document}

\maketitle

\begin{abstract}
This paper considers a class of structured cardinality penalized optimization problems with nonsmooth convex loss functions, for which the proximal mapping associated with the structured $\ell_0$ penalty is generally unavailable. To overcome this difficulty, we develop a chain of exact reformulations based on variable splitting, exact penalization, and continuous relaxation, and establish rigorous equivalence relationships among these resulting models. In particular, we define a class of strong local minimizers for the original problem and show that these minimizers are equivalent to weak $d$-stationary points of the relaxation problem.
Further, we propose a safeguarded accelerated proximal alternating direction method of multipliers (SAP-ADMM). We prove that the sequence generated by SAP-ADMM converges to a weak \(d\)-stationary point of the relaxation problem. Depending on the parameter settings and algorithmic regime, we establish $O(1/k)$ or $o(1/k)$ convergence rates for both the objective function value gap and the KKT residual. Numerical experiments on signal and image denoising demonstrate
the effectiveness of both the proposed models and SAP-ADMM in terms of reconstruction quality and computational efficiency.
\end{abstract}



\section{Introduction}
Given a matrix $A \in \mathbb{R}^{r \times n}$, a full row rank matrix $D \in \mathbb{R}^{m \times n}$ and a parameter $\lambda_0>0$, we consider the following structured cardinality penalized problem in this paper:
\begin{equation}\label{ob0}
	\min_{x\in\mathbb{R}^n} f(Ax)+\lambda_0\left\|Dx\right\|_0,\tag{SCP}
\end{equation}
where $f:\mathbb{R}^r \to \mathbb{R}$ is convex and globally $L_f$-Lipschitz continuous, but not necessarily smooth. Without loss of generality, we assume that problem~\eqref{ob0} admits a minimizer. We further assume that the proximal operator of $f$ can be computed efficiently, whereas the proximal operator of $f\circ A$ is unavailable or difficult to compute directly.
The problem \eqref{ob0} with $D=I_n$ reduces to the classical $\ell_0$ penalized model, which has been extensively studied  \cite{Blumensath2009Iterative,Lu2014Iterative,Lu2012Sparse,Ye2026S}. However, many applications involve variables with inherent ordering, such as spatial locations or temporal samples, where neighboring variables are highly correlated. For example, in image restoration, neighboring pixels are usually similar except across object boundaries; in phoneme classification, predictors from consecutive time frames lead to locally correlated classifier coefficients. In these cases, sparsity is often desired in suitable linear transformations rather than in variables themselves, then various choices of $D$ arise in applications \cite{Ti2011The}. Most existing work on problem \eqref{ob0} considers the least-squares loss \cite{J2020F,L1997v,Ti2011The}, which is sensitive to outliers and unsuitable for censored data \cite{F2001V}. In such settings, $\ell_1$, Huber, and censored regression losses are commonly used. The check loss has also been applied to array comparative genomic hybridization analysis \cite{E2005Q}. These observations motivate us to study problem \eqref{ob0} with a nonsmooth convex loss $f$ and a general linear operator $D$.

The closed-form expression of the proximal operator of $\ell_0$ penalty has enabled the development of many efficient algorithms for problem \eqref{ob0} with $D=I_n$, such as iterative hard thresholding type algorithms \cite{Bian2023A,Blumensath2008Iterative,Blumensath2009Iterative,Lu2014Iterative}, the penalty decomposition algorithm \cite{Lu2012Sparse} and Newton-type algorithms \cite{Ye2026S,Zhou2021N}. Among them, first-order methods are popular for large-scale sparse optimization due to their simplicity and low  computational cost. However, the linear operator $D$ destroys the separability of $\ell_0$ penalty, making the proximal operator of $\|D\cdot\|_0$ generally intractable. 
Although Jewell et al. \cite{J2020F} proposed a polynomial-time algorithm for the case with $D$ as the first-order difference matrix $D_1$, their algorithm strictly relies on the structure of $D_1$. To the best of our knowledge, no efficient algorithm is currently available for computing the proximal operator of $\|D\cdot\|_0$ with a general formulation of $D$. 
Hence, existing hard-thresholding-based methods cannot be directly extended to solve problem \eqref{ob0}. To overcome this difficulty, we introduce an auxiliary variable $y$ to decouple the linear map $D$ from the $\ell_0$ term, leading to the following equivalent constrained reformulation of problem \eqref{ob0}:
\begin{equation}\label{cp2}
\begin{aligned}
\min_{x\in\mathbb{R}^n, y\in\mathbb{R}^m} f(Ax)+\lambda_0\|y\|_0,
\quad \mbox{s.t.}\ y=Dx.
\end{aligned}\tag{SCP\textsubscript{C}}
\end{equation}
Replacing the linear constraint by an exact penalty term with parameter $\lambda_p>0$ yields the following unconstrained penalized problem
\begin{equation}\label{p1}
\min_{x\in\mathbb{R}^n, y\in\mathbb{R}^m} f(Ax)+\lambda_0\|y\|_0+\lambda_p\left\|Dx-y\right\|.
\tag{SCP\textsubscript{P}}
\end{equation}

Besides directly handling the $\ell_0$ penalty, a widely adopted alternative is to replace the $\ell_0$ penalty with a continuous surrogate. The $\ell_1$ norm is a popular choice due to its convexity and computational tractability \cite{Beck2009A,Candes2006Robust}, but it often yields biased estimates, especially for large coefficients. 
To alleviate this issue, various continuous nonconvex surrogates, including the $\ell_p$ $(0<p<1)$ \cite{Chen2010L}, SCAD \cite{F2001V} and capped-$\ell_1$ \cite{Pel2008A} penalties, have been proposed, which generally yield sparser solutions with smaller bias in many applications \cite{Bian2017O,Char2008R}. Among them, the capped-$\ell_1$ penalty has attracted particular attention due to its simple form, favorable theoretical properties and numerical performance \cite{Bian2020a,Le2015D,Xie2026N}. In particular, for the classical $\ell_0$ penalized problem, Bian and Chen \cite{Bian2020a} proved that the capped-$\ell_1$ relaxation is exact for global minimizers and established the equivalence between a class of local minimizers of the original problem and weak $d$-stationary points of the relaxation problem. 
Motivated by these advantages, we adopt the capped-$\ell_1$ function as the relaxation of $\ell_0$ penalty in this paper. 
However, the existing exact relaxation theory is restricted to the case $D=I_n$ and can not be directly extended to problem \eqref{ob0}.
Moreover, since problems \eqref{cp2} and \eqref{p1} can be equivalent under a suitable choice of penalty parameter (see Section \ref{fe}), it is natural to ask whether this equivalence remains valid after replacing the $\ell_0$ penalty with its capped-$\ell_1$ surrogate.
To answer these questions, we study the capped-$\ell_1$ relaxations of problems \eqref{cp2} and \eqref{p1}. Specifically, for a given parameter $\nu>0$, we define the capped-$\ell_1$ function $\phi:\mathbb{R}\to\mathbb{R}$ and its componentwise extension $\Phi:\mathbb{R}^m\to\mathbb{R}$ by 
\begin{equation}\label{pht}
\phi(t)=\min\left\{1,\frac{|t|}{\nu}\right\} \quad \mbox{and}\quad \Phi(y)=\sum_{i=1}^m\phi(y_i),
\end{equation}
where $\Phi$ serves as a continuous nonconvex relaxation of $\ell_0$ penalty.
Replacing the $\ell_0$ penalty in problem \eqref{cp2} with $\Phi$ yields the following relaxation problem:
\begin{equation}\label{rcp2}
\begin{aligned}
\min_{x\in\mathbb{R}^n, y\in\mathbb{R}^m} f(Ax)+\lambda_0\Phi(y),\quad \mbox{s.t.}\ y=Dx.
\end{aligned}\tag{SCP\textsubscript{CR}}
\end{equation}
Similarly, the corresponding relaxation of problem \eqref{p1} is
\begin{equation}\label{rp1}
\min_{x\in\mathbb{R}^n, y\in\mathbb{R}^m} f(Ax)+\lambda_0\Phi\left(y\right)+\lambda_p\left\|Dx-y\right\|.\tag{SCP\textsubscript{PR}}
\end{equation}
One of the main goals of this paper is to establish relationships among problems \eqref{cp2}, \eqref{p1}, \eqref{rcp2} and \eqref{rp1}, thereby providing a theoretical foundation for solving problem \eqref{ob0} via the capped-$\ell_1$ relaxation.

We note that a general composite optimization problem $\min_{x\in\mathbb{R}^n} f(x)+g(Dx)$ with a proper and closed function $g:\mathbb{R}^m\to\mathbb{R}\cup\{+\infty\}$, has been extensively studied. This framework covers many structured optimization models, including the fused Lasso \cite{Ti2005Sp} and the generalized Lasso \cite{Ti2011The}.
When $g$ is convex, a common strategy is to introduce an auxiliary variable and solve the resulting constrained problem by the alternating direction method of multipliers (ADMM), whose convergence is established in \cite{EC1992o,Ga1976}. If $g$ is further  Lipschitz continuous, two accelerated algorithms based on gradient descent and proximal gradient schemes were proposed in \cite{bot2020V,bot2015AV} by using Moreau envelope smoothing, and \cite{TRAN2018a} developed a smoothing primal-dual algorithm based on Bregman functions.
Recently, Sun et al. \cite{Sun2025Ac} designed an accelerated preconditioned ADMM for linearly constrained convex problems, establishing the global convergence of the iterates together with asymptotic $o(1/k)$ and nonasymptotic $O(1/k)$ convergence rates. However, these methods rely on the convexity of $g$ and can not be applied to the problems with nonconvex penalties $\ell_0$ and $\Phi$ considered here. 

Nonconvex choices of $g$ have attracted increasing attention due to their superior sparsity-inducing ability \cite{L1997v,L2019A,Par2015C,Wu2025An}, and various algorithms have been developed under different assumptions on $f$ and $D$. We note that many existing methods require $f$ to be twice or Lipschitz continuously differentiable and $D$ to be surjective or full row rank, and establish global convergence by using the Kurdyka-\L ojasiewicz (K\L) property. For example, Li and Pong studied a proximal ADMM framework \cite{Li2015Glo}; Bolte et al.  proposed a generic adaptive framework including the proximal linearized alternating minimization method \cite{Bolte2018NO}; and Bo\c{t} and Nguyen developed proximal (linearized) ADMM algorithms \cite{bot2020The}. Without relying on the K\L\  property, \cite{Li2015Glo} and \cite{Bolte2018NO} showed that any accumulation point of the iterates is a stationary point. For the case where $g$ is continuous, Liu et al. \cite{L2019A} proposed a successive difference-of-convex approximation method (SDCAM), and proved the same convergence results. SDCAM follows a double-loop framework involving an approximate solution of the inner subproblem at each outer iteration. This nested structure may increase the computational burden and implementation complexity, which motivates the development of single-loop algorithms that avoid inner iterative solves.
Therefore, existing approaches still lack global convergence guarantees and convergence rate analysis for problems with a nonsmooth loss function and nonconvex penalties.
For \eqref{ob0} with a general $D$, the relationship between the $\ell_0$ penalized model and its nonconvex relaxation is also unclear. Moreover, the $\ell_0$ penalty directly controls the number of nonzero elements, which facilitates the identification of important predictors and model interpretability \cite{Bian2020a,L1997v,Wu2025An}. These considerations motivate us to study problem \eqref{ob0} directly, and develop a first-order single-loop algorithm for it, with global convergence guarantees independent of the K\L\ property and explicit convergence rates.


To sum up, the main theoretical contributions of this paper are as follows:
\begin{itemize}
	\item We establish a theoretical framework connecting problems \eqref{ob0}, \eqref{cp2}, \eqref{p1}, \eqref{rcp2} and \eqref{rp1}. Under mild conditions, we prove that \eqref{p1} and \eqref{cp2}, as well as \eqref{rp1} and \eqref{rcp2}, share the same global and local minimizers. We also define a class of strong local minimizers of \eqref{ob0} and prove that they coincide with the weak $d$-stationary points of \eqref{rp1}. The detailed relationships among these problems are illustrated in Fig.~\ref{re}.
	\item We propose a safeguarded accelerated proximal ADMM (SAP-ADMM) algorithm that integrates a safeguarded acceleration strategy with the fast Krasnosel'ski\u{\i}--Mann (KM) and Halpern iterations. We show that the iterates generated by SAP-ADMM converge globally to a strong local minimizer of \eqref{ob0}. Moreover, depending on the parameter settings and the corresponding iteration regime, we establish that SAP-ADMM achieves either an asymptotic convergence rate of $o(1/k)$ or a nonasymptotic convergence rate of $O(1/k)$ for both the objective function value gap and the KKT residual.
\end{itemize}
\begin{figure}[htbp]
	\centering
	\includegraphics[width=\textwidth]{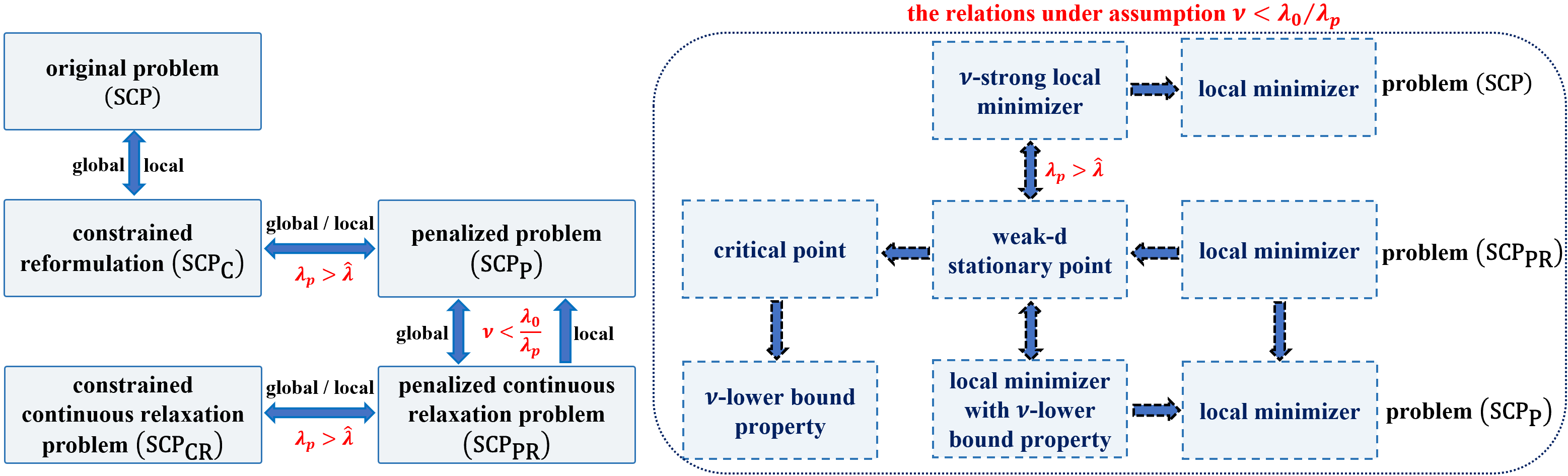}
	\caption{Relations among problems \eqref{ob0}, \eqref{cp2}, \eqref{p1}, \eqref{rcp2} and \eqref{rp1}. Here, $\hat{\lambda}:=\frac{\|A\|L_f}{\sigma_{\min}(D^{\top})}$, and “local” and “global” denote the local and global minimizers of the corresponding problems, respectively.}
	\label{re}
\end{figure}

The paper is organized as follows. In Section~\ref{sec:main}, we review some preliminaries and establish the relationships among global minimizers, local minimizers, and stationary points of the proposed models. In Section~\ref{sec:algorithm}, we propose the SAP-ADMM algorithm and analyze its convergence. Numerical results are reported in Section~\ref{sec:experiments} to demonstrate the higher-quality reconstruction of the proposed model than those based on smooth loss functions and convex penalties, and favorable trade-off between computational efficiency and reconstruction accuracy of the proposed algorithm. Finally, we conclude the paper in Section~\ref{sec:conclusions}.

\textbf{Notations} Let $\mathbb{Z}_+$ denote the set of positive integers and $\mathbb{R}^n$ be the $n$-dimensional Euclidean space. For any symmetric positive semidefinite matrix $\mathcal M\in\mathbb{R}^{n\times n}$, denote $\|x\|_{\mathcal M}^2=\langle x,\mathcal Mx\rangle$. For any $m\in \mathbb{Z}_+$, let $[m]:=\{1,2,\cdots,m\}$ and $\mathbb{D}^m:=[3]^m$.  $\|\cdot\|$ denotes the Euclidean norm.
For a set $U\subseteq \mathbb{R}^n$, let $\mathrm{co} U$ denote the convex hull of $U$. For any $\varsigma>0$, define
$B(x, \varsigma)=\{y\in\mathbb{R}^n: \|y-x\|\leq \varsigma\}.$
Let $I_m$ be the $m\times m$ identity matrix, $e_i$ its $i$-th column, and $\mathbf{0}$ the zero vector. 
For any matrix \(B\), let \(B^\top\), \(\|B\|\), and
\(\sigma_{\min}(B)\) denote its transpose, spectral norm, and smallest
singular value, respectively. For a proper closed function $h:\mathbb{R}^n\to(-\infty,+\infty]$, let $\partial h$ denote its limiting subdifferential \cite{rockafellar2009variational}, and define its proximal operator by $\operatorname{prox}_h(y)=\arg\min_{x\in\mathbb{R}^n}\{h(x)+\|x-y\|^2/2\}.$

\section{From the original problem to penalized and relaxed formulations}
\label{sec:main}
In this section, we first collect several preliminary results that will be used in the subsequent analysis. Then, we investigate the relationships among problems \eqref{cp2}, \eqref{p1}, \eqref{rcp2}, and \eqref{rp1}, thereby clarifying how the original problem is related to its penalized and relaxed formulations. 
\subsection{Preliminaries} \label{np}
We first introduce a notion of stationary point for a class of optimization problems, which will be subsequently applied to problems \eqref{p1} and \eqref{rp1}.

\begin{definition}\label{cd}
Let $h_i:\mathbb{R}^n\to(-\infty,+\infty]$, $i\in[N]$, be proper closed functions. We call $\bar{z}\in\mathbb{R}^n$ a (limiting) stationary point of problem $\min_{z\in\mathbb{R}^n}\sum_{i=1}^N h_i(z)$ if
$\mathbf{0}\in\sum_{i=1}^N\partial h_i(\bar{z})$.
\end{definition}

Any local minimizer of problems \eqref{p1} and \eqref{rp1} is a stationary point of the corresponding problem by \cite[Exercise~10.10]{rockafellar2009variational}. For notational simplicity, we define
\begin{equation}\label{pd}
g(x,y):=\|Dx-y\|,\quad \forall x\in\mathbb{R}^n,\ y \in \mathbb{R}^m.
\end{equation}
Clearly, $g(x,y)=0$ if and only if $y=Dx.$
Using \cite[Theorem 23.9]{RocC}, we obtain
\begin{equation}\label{w2}
\partial g(x, y)=\left\{
\begin{aligned}
&\begin{pmatrix}
\frac{D^{\top}(Dx-y)}{g(x,y)}\\[1.2ex]
\frac{-(Dx-y)}{g(x,y)}
\end{pmatrix},& &\quad\mbox{if}\ y\neq Dx, \\
&\left\{\begin{pmatrix}
D^{\top} \\
-I_m
\end{pmatrix}\xi:\|\xi\|\leq 1\right\},& &\quad\mbox{if}\ y= Dx.
\end{aligned}
\right.
\end{equation}

We next present some properties of $\Phi$ and some notations used throughout the paper. By \eqref{pht}, $\phi$ admits the following DC (difference-of-convex) decomposition:
\begin{equation}\label{dp}
\phi(t)=\frac{|t|}{\nu}-\max \{\theta_1(t), \theta_2(t), \theta_3(t)\}, \quad\forall t\in\mathbb{R},
\end{equation}
where $\theta_1(t)=0$, $\theta_2(t)=\frac{t}{\nu}-1$ and $\theta_3(t)=-\frac{t}{\nu}-1$. 
For any $t\in\mathbb{R}$, $y\in\mathbb{R}^m$, and $d\in \mathbb{D}^m$, define $\mathcal{D}(t)=\{i\in [3]: \theta_i(t)=\max_{j\in[3]}\theta_j(t)\}$, $\mathcal{D}^{y}=\{z\in\mathbb{D}^m: z_i\in\mathcal{D}(y_i), \forall i\in [m]\}$ and $\Phi^{d}(y)=\sum_{i=1}^m \frac{|y_i|}{\nu}- \sum_{i=1}^m \theta_{d_i}(y_i)$. The function $\Phi^{d}$ is convex for each $d\in\mathbb{D}^m$. Moreover, by \eqref{dp}, we have
\begin{equation}\label{phd}
\Phi(y)=\min_{d\in\mathbb{D}^m}\Phi^{d}(y), \quad\forall y\in\mathbb{R}^m,
\end{equation}
and can introduce the notion of weak $d$-stationary point for problem \eqref{rp1}.

\begin{definition}\cite{C2017}\label{wd_point}
	A point $(\hat{x}, \hat{y})\in\mathbb{R}^n\times\mathbb{R}^m$ is called a weak $d$(directional)-stationary point of problem \eqref{rp1} if there exist $\hat{d}_i\in\mathcal{D}(\hat{y}_i)$ for all $i\in [m]$ such that
	\begin{equation}\label{lf}
	\mathbf{0}\in 
	\begin{pmatrix}
	A^{\top}\partial f(A\hat{x}) \\
	\frac{\lambda_0}{\nu}\sum_{i=1}^m\partial |\hat{y}_i|e_i-\lambda_0\sum_{i=1}^m\theta_{\hat{d}_i}'(\hat{y}_i)e_i
	\end{pmatrix}
	+\lambda_p\partial g(\hat{x}, \hat{y}).
	\end{equation}
\end{definition}

Every local minimizer of \eqref{rp1} is a weak $d$-stationary point of it. Moreover, by Definition \ref{cd} and $\partial\!\left(
-\sum_{i=1}^m\max_{j\in[3]}\theta_j(\hat y_i)
\right)
=
\bigcup_{\hat d\in\mathcal D^{\hat y}}
\left\{
-\sum_{i=1}^m\theta_{\hat d_i}'(\hat y_i)e_i
\right\}$, the notions of weak $d$-stationary point and stationary point coincide for problem \eqref{rp1}.

\subsection{Equivalence between problems \eqref{cp2} and \eqref{p1}}\label{fe}
This subsection is devoted to establishing the equivalence between problems \eqref{cp2} and \eqref{p1} in terms of global and local minimizers. We first state the following assumption and then show every stationary point of problem \eqref{p1} is feasible for problem \eqref{cp2}.

\begin{assumption}\label{asp}
The penalty parameter $\lambda_p$ in \eqref{p1} satisfies $\lambda_p > \frac{\|A\|L_f}{\sigma_{\min}(D^{\top})}$.
\end{assumption}

\begin{lemma}\label{sp}
Under Assumption \ref{asp}, any stationary point $(x^*, y^*)$ of \eqref{p1} satisfies $g(x^*, y^*)=0$.
\end{lemma}
\begin{proof}
We prove this result by contradiction. Suppose that $g(x^*, y^*)\neq 0$. By the definition of the stationary point, it holds that
$$
\mathbf{0}\in 
\begin{pmatrix}
A^{\top}\partial f(Ax^*) \\
\lambda_0\partial \|y^*\|_0
\end{pmatrix}
+\lambda_p\partial g(x^*, y^*).
$$
By $y^*\neq Dx^*$, we have $\partial_x g(x^*,y^*)=\left\{D^\top(Dx^*-y^*)/g(x^*,y^*)\right\}$. Thus there exists $\xi^*\in\partial f(Ax^*)$ such that 
$\mathbf{0}=A^\top\xi^*+\lambda_p D^\top(Dx^*-y^*)/g(x^*,y^*)$. This yields  
\begin{equation*}\label{e2} 
\|A^\top\|\cdot\|\xi^*\|\geq\|A^\top\xi^*\|=\lambda_p \left\|\frac{D^\top(Dx^*-y^*)}{g(x^*, y^*)}\right\|  
\geq \lambda_p\sigma_{\min}(D^\top). 
\end{equation*} 
Then, we obtain
$\lambda_p  
\leq \frac{\|A^\top\|\|\xi^*\|}{\sigma_{\min}(D^\top)} 
\leq \frac{\|A^\top\|L_f}{\sigma_{\min}(D^\top)},$  
which contradicts Assumption~\ref{asp}.
\end{proof}

With the aid of the above result, we now establish the equivalence between problem \eqref{cp2} and its penalized problem \eqref{p1}.

\begin{proposition}\label{lg1}
	Under Assumption \ref{asp}, we have the following facts.
\begin{itemize}
	\item [(i)] $(x^*, y^*)$ is a local minimizer of \eqref{p1} if and only if it is a local minimizer of \eqref{cp2}, attaining the same objective value at $(x^*, y^*)$.
	
	\item [(ii)] The problems \eqref{p1} and \eqref{cp2} share the same global minimizers and optimal value.
\end{itemize}
\end{proposition}

\begin{proof}
We only prove the equivalence of local minimizers. The equivalence of global minimizers then follows by the same arguments, with the neighborhood restrictions removed.

``$\Rightarrow$" If $(x^*, y^*)$ is a local minimizer of \eqref{p1}, then $y^*=Dx^*$ by Lemma \ref{sp}, and there exists $\delta>0$ such that for all $x\in B(x^*, \delta)$ and $y\in B(y^*, \delta)$,
\begin{equation}\label{lc2}
\begin{aligned}
f(Ax^*)+\lambda_0\|y^*\|_0 
\leq f(Ax)+\lambda_0\|y\|_0+\lambda_p g(x, y).
\end{aligned}
\end{equation}
 Moreover, there exists $0<\delta'\leq\delta$ such that $Dx \in B(y^*, \delta)$ whenever $x \in B(x^*, \delta')$.
Restricting \eqref{lc2} to points satisfying $y = Dx$, we obtain that for all $x \in B(x^*, \delta')$,
$
f(Ax^*)+\lambda_0\|y^*\|_0\leq f(Ax)+\lambda_0\|Dx\|_0= f(Ax)+\lambda_0\|y\|_0.
$
This shows that $(x^*, y^*)$ is a local minimizer of \eqref{cp2}.

``$\Leftarrow$" Suppose that $(x^*, y^*)$ is a local minimizer of \eqref{cp2}. Then $y^*=Dx^*$ and there exists $\delta>0$ such that for all $x\in B(x^*, \delta)$, $y\in B(y^*, \delta)$ satisfying $y=Dx$,
\begin{equation}\label{c2}
\begin{aligned}
f(Ax^*)+\lambda_0\|y^*\|_0\leq f(Ax)+\lambda_0\|y\|_0,
\end{aligned}
\end{equation}

Denote $\hat{\delta}=\min\left\{\frac{\delta\sigma_{\min}(D^{\top})}{4\|D\|}, \frac{\delta\sigma_{\min}(D^{\top})}{4}, \frac{\delta}{2}\right\}$. 
For all $x\in B(x^*, \hat{\delta})$ and $y\in B(y^*, \hat{\delta})$, since $D$ has full row rank, we can define $\tilde{x}:=x-D^\top(DD^\top)^{-1}(Dx-y)$. Then,
\begin{equation}\label{tl3}
D\tilde{x}=Dx-DD^\top(DD^\top)^{-1}(Dx-y)=Dx-(Dx-y)=y.
\end{equation}
Moreover, we have
\begin{equation}\label{tl1}
\begin{aligned}
\|\tilde{x}-x\|=\|D^\top(DD^\top)^{-1}(Dx-y)\|
&\leq \|D^\top(DD^\top)^{-1}\|\cdot g(x, y) \\
&=\frac{1}{\sigma_{\min}(D^{\top})}g(x, y)
\end{aligned}
\end{equation}
and 
\begin{equation}\label{q1}
\begin{aligned}
g(x, y)=\|Dx-y\|
&\leq \|D\|\cdot\|x-x^*\|+\|y^*-y\| \\
&\leq \|D\|\cdot\frac{\delta\sigma_{\min}(D^{\top})}{4\|D\|}+\frac{\delta}{4}\sigma_{\min}(D^{\top})
=\frac{\delta}{2}\sigma_{\min}(D^{\top}).
\end{aligned}
\end{equation}
Based on \eqref{tl1} and \eqref{q1}, we obtain $\|\tilde{x}-x\|\leq \frac{\delta}{2}$. Hence,
$
\|\tilde{x}-x^*\|\leq \|\tilde{x}-x\|+\|x-x^*\|\leq \frac{\delta}{2}+\frac{\delta}{2}=\delta,
$
which implies that $\tilde{x}\in B(x^*, \delta).$
By the Lipschitz continuity of $f$ and \eqref{tl1}, for any $x\in B(x^*, \hat{\delta})$ and $y\in B(y^*, \hat{\delta})$, it holds that
\begin{equation*}
\begin{aligned}
\left|f(A\tilde{x})-f(Ax)\right|
\leq L_f\|A\|\|x-\tilde{x}\| 
\leq \frac{L_f\|A\|}{\sigma_{\min}(D^{\top})}g(x, y),
\end{aligned}
\end{equation*}
which means that 
\begin{equation*}
\begin{aligned}
f(Ax)+\lambda_0\|y\|_0
&\geq f(A\tilde{x})+\lambda_0\|y\|_0-\frac{L_f\|A\|}{\sigma_{\min}(D^{\top})}g(x, y) \\
&\geq f(Ax^*)+\lambda_0\|y^*\|_0-\frac{L_f\|A\|}{\sigma_{\min}(D^{\top})}g(x, y),
\end{aligned}
\end{equation*}
where the last inequality follows from $\tilde{x}\in B(x^*, \delta)$, \eqref{c2} and \eqref{tl3}.
Consequently, for any $x\in B(x^*, \hat{\delta})$ and $y\in B(y^*, \hat{\delta})$, 
\begin{equation*}
\begin{aligned}
&f(Ax)+\lambda_0\|y\|_0+\lambda_pg(x, y) \\
\geq& f(Ax^*)+\lambda_0\|y^*\|_0+\left(\lambda_p-\frac{L_f\|A\|}{\sigma_{\min}(D^{\top})}\right)g(x, y) \\
\geq &f(Ax^*)+\lambda_0\|y^*\|_0 
=f(Ax^*)+\lambda_0\|y^*\|_0+\lambda_pg(x^*, y^*),
\end{aligned}
\end{equation*}
where the second inequality holds due to Assumption \ref{asp}, while the last equality uses $Dx^*=y^*$.
Therefore, $(x^*, y^*)$ is a local minimizer of \eqref{p1}.

Based on the above arguments, we conclude that problems \eqref{p1} and \eqref{cp2} attain the same objective function value at the point $(x^*, y^*)$.
\end{proof}

We conclude this subsection by providing an equivalent characterization of the local minimizers of problem \eqref{p1}.

\begin{remark}\label{ro3}
	$(\hat{x}, \hat{y})$ is a local minimizer of problem \eqref{p1} if and only if $(\hat{x}, \hat{y})$ is a global minimizer of $f(Ax)+\lambda_pg(x,y)$ on $\mathbb{R}^n\times \mathbb{R}_s^m$, which is equivalent to
\begin{equation}\label{w1}
\mathbf{0}\in 
\begin{pmatrix}
A^{\top}\partial f(A\hat{x}) \\
0
\end{pmatrix}
+\lambda_p\partial g(\hat{x}, \hat{y})
+\begin{pmatrix}
0\\
\mathcal{N}_{\mathbb{R}_s^m} (\hat{y})
\end{pmatrix},
\end{equation}
where $\mathbb{R}_s^m=\{y\in\mathbb{R}^m: y_i=0, \forall i\in\mathcal{I}(\hat{y})\}$ with $\mathcal{I}(\hat{y})=\{i\in[m]: \hat{y}_i=0\}$, and $\mathcal{N}_{\mathbb{R}_s^m} (\hat{y})=\{z\in\mathbb{R}^m: z_i\in\mathbb{R}, i\in \mathcal{I}(\hat{y}); z_i=0, i\notin \mathcal{I}(\hat{y})\}$.
\end{remark}

\subsection{Relationships among problems \eqref{rcp2}, \eqref{rp1} and \eqref{p1}}
In this subsection, we first establish the equivalence between problems \eqref{rcp2} and \eqref{rp1}. We then investigate the relationships between problems \eqref{rp1} and \eqref{p1}. To facilitate the subsequent analysis, we first recall a weaker stationarity notion for problem \eqref{rp1}, namely, a critical point.
\begin{definition}\cite{C2017}
$(x^*, y^*)$ is called a critical point of \eqref{rp1} if it satisfies 
\begin{equation*}
\mathbf{0}\in 
\begin{pmatrix}
A^{\top}\partial f(Ax^*) \\
\frac{\lambda_0}{\nu}\sum_{i=1}^m\partial|y_i^*|e_i
\end{pmatrix}
-\begin{pmatrix}
0 \\
\lambda_0\partial (\sum_{i=1}^m\max_{j\in[3]}\theta_j(y^*_i))
\end{pmatrix}+\lambda_p\partial g(x^*, y^*).
\end{equation*}
\end{definition}
Clearly, every weak $d$-stationary point is a critical point for \eqref{rp1}. Using arguments similar to those in 
Lemma~\ref{sp} and Proposition~\ref{lg1}, we can establish the following results. The proof is therefore omitted.

\begin{proposition}\label{lg2}
Under Assumption \ref{asp}, the following statements hold.
	\begin{itemize}
		\item [(i)] If $(x^*, y^*)$ is a critical point of problem \eqref{rp1}, then $y^*=Dx^*$.
		\item [(ii)] $(x^*, y^*)$ is a global (resp. local) minimizer of \eqref{rcp2} if and only if it is a global (resp. local) minimizer of \eqref{rp1}, and both problems achieve the same objective function value at $(x^*, y^*)$.
	\end{itemize}
\end{proposition}

Next, we investigate the relationship between problems \eqref{rp1} and \eqref{p1}. We assume that the relaxation parameter satisfies the following condition.

\begin{assumption}\label{asn}
	The relaxation parameter $\nu$ in \eqref{pht} satisfies $\nu<\frac{\lambda_0}{\lambda_p}$.
\end{assumption}

 We show that every critical point of \eqref{rp1} enjoys a lower bound property, i.e., the absolute value of each nonzero component is bounded away from zero.

\begin{lemma}\label{lw}
Under Assumption \ref{asn}, if $(x^*, y^*)$ is a critical point of problem \eqref{rp1}, then for each $i\in[m]$, either $y_i^*= 0$ or $|y_i^*|\geq\nu$.
\end{lemma}
\begin{proof}
We proceed by contradiction. Suppose that there exists $\hat i\in[m]$ such that $|y_{\hat{i}}^*|\in (0,\nu)$. 
By Proposition \ref{lg2}, we have $y^*=Dx^*$. 
Therefore, there exist $\zeta_i^*\in\mathrm{co}\{\theta'_j(y_i^*): j\in\mathcal{D}(y_i^*)\}$ for all $i\in[m]$ and $\xi^*\in\mathbb{R}^m$ with $\|\xi^*\|\leq 1$ such that 
$\mathbf{0}\in \frac{\lambda_0}{\nu}\sum_{i=1}^m\partial|y_i^*|e_i-\lambda_0\sum_{i=1}^m\zeta_i^*e_i-\lambda_p\xi^*.$
Since $|y_{\hat{i}}^*|\in (0,\nu)$, we have $\mathcal{D}(y_{\hat{i}}^*)=\{1\}$. Then,
$0= \frac{\lambda_0}{\nu}\mathrm{sign}(y_{\hat{i}}^*)-\lambda_p\xi^*_{\hat{i}},$
which implies that
$\frac{\lambda_0}{\nu}=\lambda_p\left|\xi^*_{\hat{i}}\right|
\leq \lambda_p\|\xi^*\|
\leq \lambda_p$.
Hence, $\frac{\lambda_0}{\lambda_p}\leq \nu$. This contradicts Assumption \ref{asn}.
\end{proof}

Based on the above lower bound property and using an argument similar to that in Theorem 2.4 of \cite{Bian2020a}, we can show that problems \eqref{rp1} and \eqref{p1} have the same global minimizers and optimal value. For brevity, the proof is omitted.

\begin{proposition}
Under Assumption \ref{asn}, a point $(x^*, y^*)$ is a global minimizer of \eqref{rp1} if and only if it is a global minimizer of \eqref{p1}. Moreover, the problems attain the same minimum at point $(x^*, y^*)$. 
\end{proposition}

To establish the relationship between the local minimizers of problems \eqref{rp1} and \eqref{p1}, we first show that the index $\hat{d}_i$ associated with a weak $d$-stationary point of problem \eqref{rp1} is uniquely determined.

\begin{lemma}\label{lp}
	Under Assumption \ref{asn}, if $(\hat{x}, \hat{y})$ is a weak $d$-stationary point of problem \eqref{rp1}, then 
	for all $i\in[m]$, the index $d^{\hat{y}}_i\in\mathcal{D}(\hat{y}_i)$ satisfying \eqref{lf} is uniquely determined and given by
	\begin{equation}\label{lifi}
	d^{\hat{y}}_i=\left\{
	\begin{aligned}
	&1,&\quad &\mbox{if}\ |\hat{y}_i|<\nu,& \\
	&2,&\quad &\mbox{if}\ \hat{y}_i \geq \nu,& \\
	&3,&\quad &\mbox{if}\ \hat{y}_i \leq -\nu.&
	\end{aligned}\right.
	\end{equation}
\end{lemma}

\begin{proof}
By the definition of $\mathcal{D}(\hat{y}_i)$, \eqref{lifi} is immediate when $|\hat{y}_i|\neq \nu$. It therefore suffices to consider the case $\hat{y}_i=\nu$, since the case $\hat{y}_i=-\nu$ is analogous.
By (i) of Proposition \ref{lg2}, we have $\hat{y}=D\hat{x}$. Then there exists $\hat{\xi}\in\mathbb{R}^m$ with $\|\hat{\xi}\|\leq 1$ such that
\begin{equation}\label{sp1}
\mathbf{0}\in 
\begin{pmatrix}
A^{\top}\partial f(A\hat{x}) \\
\frac{\lambda_0}{\nu}\sum_{i=1}^m\partial|\hat{y}_i|e_i-\lambda_0\sum_{i=1}^m\theta_{d^{\hat{y}}_i}'(\hat{y}_i)e_i
\end{pmatrix}
+\lambda_p\begin{pmatrix}
D^{\top}\hat{\xi} \\
-\hat{\xi}
\end{pmatrix}.
\end{equation}
Since $\hat{y}_i=\nu$, it follows that $\mathcal{D}(\hat{y}_i)=\{1, 2\}$. We prove by contradiction that $d^{\hat{y}}_i\neq 1$. Suppose that \eqref{sp1} holds with $d^{\hat{y}}_i=1$. By the corresponding $i$-th component of \eqref{sp1}, we have
$0= \frac{\lambda_0}{\nu}-\lambda_p\hat{\xi}_i.$
Consequently, $\frac{\lambda_0}{\nu}=\lambda_p\hat{\xi}_i\leq \lambda_p\| \hat{\xi}\|\leq \lambda_p$. Thus $\frac{\lambda_0}{\lambda_p}\leq \nu,$
which contradicts the choice of $\nu$. Hence, $d^{\hat{y}}_i=2$. 
\end{proof}

\begin{proposition}\label{ro0}
Under Assumptions \ref{asp} and \ref{asn}, $(\hat{x}, \hat{y})$ is a local minimizer of \eqref{p1} satisfying $|\hat{y}_i|\notin(0,\nu)$ for all $i\in[m]$ if and only if it is a weak $d$-stationary point of \eqref{rp1}, and two problems have the same objective function value at $(\hat{x}, \hat{y})$.
\end{proposition}

\begin{proof}
Assume that $(\hat{x}, \hat{y})$ is a local minimizer of \eqref{p1} satisfying $|\hat{y}_i|\notin(0,\nu)$ for all $i\in[m]$. 
By \eqref{w1} and \eqref{w2}, there exists $\hat{\xi}\in\mathbb{R}^m$ with $\|\hat{\xi}\|\leq 1$ such that
\begin{equation}\label{sa}
\mathbf{0}\in A^{\top}\partial f(A\hat{x})+\lambda_pD^{\top}\hat{\xi}\quad\mbox{and}\quad \lambda_p\hat{\xi}\in \mathcal{N}_{\mathbb{R}_s^m} (\hat{y}),
\end{equation}
where $\mathbb{R}_s^m$ is defined as in \eqref{w1}.
We next show that there exists $d^{\hat{y}}\in\mathcal{D}^{\hat{y}}$ such that \eqref{sp1} holds for such $\hat{\xi}$. To this end, we construct $d^{\hat{y}}$ elementwise.
Since $|\hat{\xi}_i|\leq \|\hat{\xi}\|\leq 1<\frac{\lambda_0}{\lambda_p\nu}$, we have $\lambda_p|\hat{\xi}_i|<\frac{\lambda_0}{\nu}$, and hence $0\in \frac{\lambda_0}{\nu}\partial |0|-\lambda_p\hat{\xi}_i$. Thus, when $\hat{y}_i=0$, we set $d^{\hat{y}}_i=1$. Then $d^{\hat{y}}_i\in\mathcal{D}(\hat{y}_i)$ and $0\in \frac{\lambda_0}{\nu}\partial |\hat{y}_i|
-\lambda_0\theta'_{d^{\hat{y}}_i}(\hat{y}_i)
-\lambda_p\hat{\xi}_i$. When $|\hat{y}_i|\geq\nu$, it holds that $\hat{\xi}_i=0$. 
Then we set $d^{\hat{y}}_i=2$ if $\hat{y}_i\geq\nu$ and $d^{\hat{y}}_i=3$ if $\hat{y}_i\leq-\nu$. In either case, $d^{\hat{y}}_i\in\mathcal{D}(\hat{y}_i)$ and $0\in \frac{\lambda_0}{\nu}\mbox{sign}(\hat{y}_i)-\lambda_0\theta'_{d^{\hat{y}}_i}(\hat{y}_i)-\lambda_p\hat{\xi}_i$.
Consequently, defining $d^{\hat{y}}=(d^{\hat{y}}_1,\cdots,d^{\hat{y}}_m)^{\mathrm{T}}$, we have $d^{\hat{y}}\in\mathcal{D}^{\hat{y}}$ and \eqref{sp1} holds for $d^{\hat{y}}$ and $\hat{\xi}$.
This implies that $(\hat{x}, \hat{y})$ is a weak $d$-stationary point of \eqref{rp1}.

On the other hand, if $(\hat{x}, \hat{y})$ is a weak $d$-stationary point of \eqref{rp1}, $\hat{y}$ satisfies the lower bound property, and there exists $\hat{d}\in\mathcal{D}^{\hat{y}}$ such that \eqref{lf} holds. By the convexity of all functions involved in \eqref{lf}, it follows that
 \begin{equation}\label{sp3}
(\hat{x}, \hat{y})\in \arg\min_{x\in\mathbb{R}^n, y\in\mathbb{R}^m} f(Ax)+\lambda_0\Phi^{\hat{d}}\left(y\right)+\lambda_p g(x, y).
 \end{equation}
By item (i) of Proposition \ref{lg2} and Lemma \ref{lw}, we have 
\begin{equation}\label{t1}
\hat{y}=D\hat{x}\quad\mbox{and}\quad\Phi(\hat{y})=\|\hat{y}\|_0.
\end{equation}
Combining \eqref{sp3},  $\Phi(\hat{y})=\Phi^{\hat{d}}(\hat{y})$ and \eqref{t1}, we obtain for all $ x\in\mathbb{R}^n$ and $y\in\mathbb{R}^m$,
\begin{equation}\label{t2}
\begin{aligned}
f(Ax)+\lambda_0\Phi^{\hat{d}}(y)+\lambda_p g(x, y)
&\geq f(A\hat{x})+\lambda_0\Phi^{\hat{d}}(\hat{y})+\lambda_p g(\hat{x}, \hat{y}) \\
&=f(A\hat{x})+\lambda_0\Phi(\hat{y})+\lambda_p g(\hat{x}, \hat{y}) \\
&=f(A\hat{x})+\lambda_0\|\hat{y}\|_0
+\lambda_p g(\hat{x}, \hat{y}).
\end{aligned}
\end{equation}
Define $\hat{I}=\{i\in [m]: |\hat{y}_i|\geq\nu\}$. For $j\notin \hat{I}$, we have $\hat{d}_j=1$. Hence, for any $y\in\mathbb{R}^m$, 
$\Phi^{\hat{d}}(y)
=\sum_{j\in \hat{I}}\left(\frac{|y_j|}{\nu}-\theta_{\hat{d}_j}(y_j)\right)+\sum_{j\notin \hat{I}}\frac{|y_j|}{\nu}.$
By \eqref{lifi}, it follows that $\hat{d}_j=2$ or $3$ for $j\in \hat{I}$. Thus there exists a $\delta>0$ such that
\begin{equation}\label{t3}
\begin{aligned}
\Phi^{\hat{d}}(y)
=&\sum_{j\in \hat{I}}1+\sum_{j\notin \hat{I}}\frac{|y_j|}{\nu}=\|\hat{y}\|_0+\sum_{j\notin \hat{I}}\frac{|y_j|}{\nu}\leq \|y\|_0,\quad y\in B(\hat{y},\delta),
\end{aligned}
\end{equation}
where the second equality uses Lemma \ref{lw}. 
For all $x\in B(\hat{x},\delta)$ and $y\in B(\hat{y},\delta)$, it follows from \eqref{t2} and \eqref{t3} that
$f(A\hat{x})+\lambda_0\|\hat{y}\|_0+\lambda_p g(\hat{x}, \hat{y})
\leq f(Ax)+\lambda_0\|y\|_0+\lambda_p g(x,y).$
Therefore, $(\hat{x},\hat{y})$ is a local minimizer of \eqref{p1}.

Based on the above analysis and \eqref{t1}, it is clear that the objective function values of the two problems coincide at the point $(\hat{x}, \hat{y})$.
\end{proof}

In view of Proposition \ref{ro0}, every local minimizer of \eqref{rp1} is also a local minimizer of \eqref{p1}. Moreover, Proposition \ref{ro0} together with the equivalence between the local minimizers of \eqref{p1} and \eqref{ob0} motivates the following definition of a strong local minimizer for problem \eqref{ob0}.

\begin{definition}
We call $\hat{x}$ a $\nu$-strong local minimizer of problem \eqref{ob0} if $\hat{x}$ is a local minimizer of problem \eqref{ob0} and $\hat{y}:=D\hat{x}$ satisfies $|\hat{y}_i|\notin(0,\nu)$ for all $i\in[m]$.
\end{definition}

By Propositions \ref{lg1} and \ref{ro0}, we directly obtain the following equivalence between $\nu$-strong local minimizers of \eqref{ob0} and weak $d$-stationary points of \eqref{rp1}.

\begin{corollary}
	Under Assumptions \ref{asp} and \ref{asn}, $\hat{x}$ is a $\nu$-strong local minimizer of \eqref{ob0} if and only if $(\hat{x}, D\hat{x})$ is a weak $d$-stationary point of \eqref{rp1}.
\end{corollary}

\section{A safeguarded accelerated proximal ADMM}
\label{sec:algorithm}

In this section, we develop an ADMM-type method to find a weak $d$-stationary point of problem \eqref{rp1}. Define the stacked variable
\(
u=(x^{\top},y^{\top})^{\top}
\)
and the linear operators
\[
    B_1=(A,0),\qquad
    B_2=(D,-I),\qquad
    B_3=(0,I).
\]
Then problem~\eqref{rp1} can be equivalently reformulated as
\begin{equation}
    \begin{aligned}
    \min_{p,q,u}\quad & f(p)+\lambda_p\|q\|+\lambda_0\Phi(B_3u)\\
    \mathrm{s.t.}\quad & B_1u-p=\mathbf{0},\quad B_2u-q=\mathbf{0}.
    \end{aligned}
    \label{eq:constrained_split_model}
\end{equation}
For any \(\rho>0\) and $d\in\mathbb{D}^m$, motivated by the augmented Lagrangian of \eqref{eq:constrained_split_model} and the convexity of $\Phi^d$, we define $\mathcal L_\rho^d:\mathbb{R}^r\times \mathbb{R}^m\times \mathbb{R}^{n+m} \times \mathbb{R}^{r}\times \mathbb{R}^{m}\to \mathbb{R}\cup\{+\infty\}$ by
\begin{equation*}
\begin{aligned}
    \mathcal L_\rho^d(p,q,u;\eta,\mu)
    &=f(p)+\lambda_p\|q\|+\lambda_0\Phi^d(B_3u)
    +\langle\eta,B_1u-p\rangle+\langle\mu,B_2u-q\rangle\\
    &\quad+\rho\|B_1u-p\|^2/2+\rho\|B_2u-q\|^2/2,
\end{aligned}
\label{eq:aug_lagrangian}
\end{equation*}
which serves as the basis for constructing the subproblems of the proposed algorithm. For simplicity, let \(w:=(p,q,u,\eta,\mu)\) and $\mathbb{W}:=\mathbb{R}^r\times \mathbb{R}^m\times \mathbb{R}^{n+m} \times \mathbb{R}^{r}\times \mathbb{R}^{m}$. Inspired by the acceleration scheme used in \cite{Sun2025Ac},
we develop a safeguarding strategy to determine whether
$w^{k+1}$ is updated by \eqref{ac} or \eqref{nac}. Incorporating this strategy into the ADMM framework and proximal methods yields
Algorithm~\ref{alg:prox_admm} to solve problem
\eqref{eq:constrained_split_model}, which is named
the \emph{Safeguarded Accelerated Proximal ADMM} (SAP-ADMM).

\begin{algorithm}[htbp]
\caption{Safeguarded Accelerated Proximal ADMM (SAP-ADMM)}
\label{alg:prox_admm}
\begin{algorithmic}
\State \textbf{Input:} Choose $\alpha\geq 2$, $t\in(0,2]$, $N_{\mathrm{max}}\in\mathbb{Z}_+$, and $\rho,\beta>0$ satisfying $\beta>\rho$ and
$\beta I-\rho A^\top A-\rho\beta D^\top D/(\beta-\rho)\succ0$. Initialize $w^0=(p^0,q^0,u^0,\eta^0,\mu^0)\in\mathbb{W}$ with
$u^0:=((x^0)^\top,(y^0)^\top)^\top$,
$x^0\in\mathbb{R}^n$, $y^0\in\mathbb{R}^m$ and
$\|\mu^0\|<\lambda_p$.

Set $d^0:=d^{y^0}$ according to~\eqref{lifi}, $\hat w^0=w^0$, $N_{\mathrm{fail}}=0$, $\mathtt{acc}=1$, and $k=0$.

\While{a stopping criterion is not satisfied}
\State \textbf{Step 1.} Compute
{\setlength{\abovedisplayskip}{2pt}
 \setlength{\belowdisplayskip}{2pt}
\begin{align}
    \bar p^k
    &=\operatorname{prox}_{f/\rho}(Ax^k+\eta^k/\rho),
    \label{eq:p_update}\\
    \bar\eta^k
    &=\eta^k+\rho(Ax^k-\bar p^k),
    \label{eq:eta_update}\\
    \bar q^k
    &=\operatorname{prox}_{\lambda_p\|\cdot\|/\rho}
    (Dx^k-y^k+\mu^k/\rho),
    \label{eq:q_update}\\
    \bar\mu^k
    &=\mu^k+\rho(Dx^k-y^k-\bar q^k).
    \label{eq:mu_update}
\end{align}
}
    \State \textbf{Step 2.} Update $\zeta^k=\rho(Dx^k-y^k-\bar q^k)+\bar\mu^k$ and
    {\setlength{\abovedisplayskip}{2pt}
 \setlength{\belowdisplayskip}{2pt}
   \begin{equation*}
   \begin{aligned}
    \bar x^k
    &=x^k-\left(
    A^\top\big(\rho(Ax^k-\bar p^k)+\bar\eta^k\big)
    +D^\top\zeta^k
    \right)/\beta,\\
    \bar y^k
    &=\operatorname{prox}_{\lambda_0\Phi^{d^k}/\beta}
    \left(y^k+\zeta^k/\beta\right).
    \end{aligned}
    \end{equation*}
    }

    \State \textbf{Step 3.} Update $d^{k+1}:=d^{\bar y^k}$ by~\eqref{lifi}. Set
{\setlength{\abovedisplayskip}{2pt}
\setlength{\belowdisplayskip}{2pt}
 \begin{equation*}
 \bar u^k=\big((\bar x^k)^\top,(\bar y^k)^\top\big)^\top,\qquad
 \bar w^k=(\bar p^k,\bar q^k,\bar u^k,\bar\eta^k,\bar\mu^k).
 \end{equation*}
}

    \State \textbf{Step 4.} Perform the safeguarded acceleration:
    \IndentedIf{$(\alpha,t)=(2,1)$ or $\mathtt{acc}=0$}
        compute
        \AlgorithmEquation{nac}{%
        w^{k+1}=\frac{1}{k+2}w^0+\frac{k+1}{k+2}\bar w^k;
        }
        \State let $k\leftarrow k+1$ and return to \textbf{Step~1}.
    \Else\ compute {$\hat w^{k+1}=(1-t)w^k+t\bar w^k$} and
        \AlgorithmEquation{ac}{%
        w^{k+1}=w^k+\frac{\alpha}{2(k+\alpha)}(\hat w^{k+1}-w^k)
        +\frac{k}{k+\alpha}(\hat w^{k+1}-\hat w^k).
        }
        \State Repeat \textbf{Steps~1--3} with $k$ replaced by $k+1$
        to obtain $\bar w^{k+1}$ and $d^{k+2}$.
        \If{$d^{k+2}\neq d^{k+1}$}
            \State $N_{\mathrm{fail}}\leftarrow N_{\mathrm{fail}}+1$, $\hat w^{k+1}=\bar w^k$; if $N_{\mathrm{fail}}\geq N_{\mathrm{max}}$, set $\mathtt{acc}=0$.
            \State Compute \eqref{nac} to obtain $w^{k+1}$; $k\leftarrow k+1$ and return to \textbf{Step~1}.
        \Else
            \State $k\leftarrow k+1$ and return to \textbf{Step~4}.
        \EndIf
    \EndIf
\EndWhileAfterIndent
\end{algorithmic}
\end{algorithm}

Compared with the classical ADMM, Algorithm~\ref{alg:prox_admm} updates $\bar y^k$ by replacing the nonconvex function $\Phi$ with its convex surrogate $\Phi^{d^k}$, where $d^k$ is determined by $\bar y^{k-1}$. As a result, the
$\bar y^k$-subproblem admits a closed-form solution (see
Remark~\ref{rem:y_prox}). Step 4 implements a safeguarded acceleration strategy. At iteration $k$, $w^{k+1}$ is updated by \eqref{ac} if
$d^{k+2}=d^{k+1}$; otherwise, it is updated by \eqref{nac}
and $N_{\mathrm{fail}}$ increases by one.
Since $N_{\mathrm{fail}}\leq N_{\mathrm{max}}$, there exists $K\in\mathbb{Z}_+$ such that $w^{k+1}$ is updated by
\eqref{ac} or \eqref{nac}
for all $k\geq K$.

\begin{remark}\label{rem:y_prox}
The proximal operator arising in the $\bar y^k$-subproblem has a closed-form expression. Specifically, define $\tau^k=y^k+\frac{1}{\beta}\zeta^k$, and for each
\(j\in[m]\), \(\bar y_j^k\) can be computed by
\begin{equation}\label{prox_y_1}
\begin{aligned}
\bar y_j^k=
\begin{cases}
0, & |\bar\tau_j^k|\le \lambda_0/(\beta\nu),\\
\bar\tau_j^k-\lambda_0/(\beta\nu), & \bar\tau_j^k>\lambda_0/(\beta\nu),\\
\bar\tau_j^k+\lambda_0/(\beta\nu), & \bar\tau_j^k<-\lambda_0/(\beta\nu),
\end{cases}
\text{where }\ 
\bar\tau_j^k=
\begin{cases}
\tau_j^k, & d_j^k=1,\\
\tau_j^k+\lambda_0/(\beta\nu), & d_j^k=2,\\
\tau_j^k-\lambda_0/(\beta\nu), & d_j^k=3.
\end{cases}
\end{aligned}
\end{equation}

\end{remark}

To analyze the convergence of Algorithm~\ref{alg:prox_admm}, we impose the following assumption throughout the remainder of the paper. Clearly, Assumption~\ref{asn_1} implies Assumption~\ref{asn}. Let \(\{w^k\}\), \(\{\bar w^k\}\) and \(\{\hat w^k\}\) be the sequences generated by Algorithm~\ref{alg:prox_admm}. The next lemma establishes that $\{d^k\}$ changes only finitely many times.

\begin{assumption}\label{asn_1}
	The relaxation parameter $\nu$ in \eqref{pht} satisfies $\nu<\frac{\lambda_0}{3\lambda_p}$.
\end{assumption}

\begin{lemma}\label{unchanged}
Under Assumption~\ref{asn_1}, there exists $\hat K\in\mathbb{Z}_+$ such that, for every $k>\hat K$, the following hold:
\begin{enumerate}
    \item[\rm(i)] if \(w^k\) is updated by \eqref{nac}, then, for any \(j\in [m]\), \(\bar y_j^{k-1}=0\) implies \(\bar y_j^k=0\), while \(0<|\bar y_j^{k-1}|<\nu\) implies either \(\bar y_j^k=0\) or \(|\bar y_j^k|\leq|\bar y_j^{k-1}|-(\lambda_0/(2\nu\beta)-3\lambda_p/(2\beta))\);
 \item[\rm(ii)] there exists a vector \(\bar d\in\mathbb D^m\) such that \(d^k=\bar d\).
\end{enumerate}
\end{lemma}

\begin{proof} 
{\rm(i)} Using \eqref{eq:q_update} and \eqref{eq:mu_update}, we obtain \(\bar\mu^k\in\lambda_p\partial\|\bar q^k\|\). Hence, $\|\bar\mu^k\|\leq\lambda_p$ for all $k\in\mathbb{Z}_+$.
Define \(\varepsilon=\lambda_0/(\nu\beta)-3\lambda_p/\beta\). By Assumption~\ref{asn_1}, \(\varepsilon>0\). There exists \(K_1>1\) such that 
\begin{equation}\label{nad1}
\|y^0\|/(k+1)<\min\{\varepsilon/2,\nu/2\},\quad\forall k>K_1.
\end{equation}
If \(w^k\) is updated by \eqref{nac}, we have \(\mu^k=\mu^0/(k+1)+k\bar\mu^{k-1}/(k+1)\). This together with \(\|\mu^0\|<\lambda_p\) gives \(\|\mu^k\|\leq\lambda_p\). By \eqref{eq:mu_update}, it holds that \(\zeta^k=2\bar\mu^k-\mu^k\). Then, \(\|\zeta^k\|\leq3\lambda_p\) when \(w^k\) is updated by \eqref{nac}.

Note that \(d_j^k=1\) when \(|\bar y_j^{k-1}|<\nu\) with any $j\in[m]$.
 
If \(\bar y_j^{k-1}=0\), \(|y_j^k|=|y_j^0/(k+1)+k\bar y_j^{k-1}/(k+1)|\leq\|y^0\|/(k+1)<\varepsilon/2\). 
By \eqref{prox_y_1}, 
$|\bar\tau_j^k|=|y_j^k+\zeta_j^k/\beta|<\varepsilon/2+3\lambda_p/\beta<\lambda_0/(\beta\nu).$
Hence, \(\bar y_j^k=0\).

When \(0<|\bar y_j^{k-1}|<\nu\), we only prove the case \(0<\bar y_j^{k-1}<\nu\), since the argument for \(-\nu<\bar y_j^{k-1}<0\) is analogous. By $y_j^k=y_j^0/(k+1)+k\bar y_j^{k-1}/(k+1)$ and \eqref{nad1}, we have \(-\varepsilon/2<y_j^k<\bar y_j^{k-1}+\varepsilon/2\) for \(k>K_1\). Using \(\|\zeta^k\|\leq3\lambda_p\), we obtain \(\bar\tau_j^k=y_j^k+\zeta_j^k/\beta>-\varepsilon/2-3\lambda_p/\beta>-\lambda_0/(\beta\nu)\). 
Then \(\bar y_j^k\geq0\) by \eqref{prox_y_1}. If \(\bar y_j^k>0\), \(\bar y_j^k=y_j^k+\zeta_j^k/\beta-\lambda_0/(\beta\nu)\leq y_j^k+3\lambda_p/\beta-\lambda_0/(\beta\nu)<\bar y_j^{k-1}-\varepsilon/2\).

{\rm(ii)} Notice that there exists $K\geq K_1$ such that
Algorithm~\ref{alg:prox_admm} updates $w^{k+1}$ by
\eqref{ac} or \eqref{nac} for all $k\geq K$.
If $w^{k+1}$ is updated by \eqref{ac} for all $k\geq K$,
then \(d^{k+2}=d^{k+1}\) for all \(k\geq K\).
Hence, \(d^k=d^{K+1}\) for all \(k>K+1\).  

Otherwise, \(w^k\) is updated by \eqref{nac}
for all \(k\geq K+1\). By \rm(i) of this lemma, it holds that for every \(k\geq K\), \(d_j^k=1\) implies \(d_j^{k+1}=1\). Since $[m]$ is finite, it suffices to show that, for every \(k>K_1\) and \(j\in [m]\), \(d_j^k=2\) implies \(d_j^{k+1}\neq3\), and \(d_j^k=3\) implies \(d_j^{k+1}\neq2\). We only prove the former, as the latter follows analogously. If \(d_j^k=2\), we have \(\bar y_j^{k-1}\geq\nu\), which together with \eqref{nad1} implies that for any \(k>K_1\), 
\(y_j^k=y_j^0/(k+1)+k\bar y_j^{k-1}/(k+1)
    >\nu/2-\|y^0\|/(k+1)
    >0.\)
By \(\|\zeta^k\|\leq3\lambda_p\) in this case, we obtain \(\bar\tau_j^k=y_j^k+\zeta_j^k/\beta+\lambda_0/(\beta\nu)\geq y_j^k+\varepsilon>0.\) Combining this with \eqref{prox_y_1}, it follows that \(\bar y_j^k\geq0\), and hence \(d_j^{k+1}\neq3\) by \eqref{lifi}. Thus, there exist \(\hat K\geq K\) and \(\bar d\in\mathbb D^m\) such that \(d^k=\bar d\) for all \(k>\hat K\).

 Then both statements {\rm(i)} and {\rm(ii)} hold for all \(k> \hat K\). 
\end{proof}

In view of Lemma~\ref{unchanged}, we know that there exist \(\bar d\in\mathbb D^m\) and \(\hat K\in\mathbb{Z}_+\) such that \(\Phi^{d^k}=\Phi^{\bar d}\) for all \(k>\hat K\). Define 
\[
h(u):=\lambda_0\Phi^{\bar d}(B_3u),\quad
\forall u\in\mathbb{R}^{n+m}
\quad\text{and}\quad
c(v):=\lambda_p\|v\|,\quad
\forall v\in\mathbb{R}^{m}.
\]
Algorithm~\ref{alg:prox_admm} reduces, after iteration $\hat K$, to solving the convex optimization problem
\begin{equation}
    \min_{u}\; f(B_1u)+c(B_2u)+h(u).
    \label{convex_pb}
\end{equation}
The following assumption ensures the existence of a minimizer of problem~\eqref{convex_pb}.

\begin{assumption}\label{assumption_4}
The function $f$ is level-bounded.
\end{assumption}

\begin{lemma}\label{existence_convex}
Under Assumption~\ref{assumption_4}, the problem~\eqref{convex_pb} admits a minimizer.
\end{lemma}

\begin{proof}
Define $S=\{(z_1,z_2):\exists\,u\ \text{with}\ z_1=B_1u,\ z_2=B_2u\}$, which is a closed convex subspace. Denote \(z=(z_1,z_2)\) and \(\widetilde f(z)=f(z_1)+c(z_2)\). Since \(f\) is level-bounded, so is \(\widetilde f\). For every \(z\in S\), define
$\widetilde h(z):=\inf\{h(u):B_1u=z_1,\ B_2u=z_2\}.$
Since \(h\) is polyhedral convex and bounded below by \(0\), it follows from \cite[Corollary~19.3.1]{RocC} that \(\widetilde h\) is polyhedral convex, lower semicontinuous, and its
infimum in its definition is attained. By \(\widetilde h(z)\geq0\) for any $z$, we obtain that $\widetilde f+\widetilde h$ is level-bounded. Consequently, the problem \(\min_{z\in S}\widetilde f(z)+\widetilde h(z)\) admits a minimizer \(z^*\). Then, there exists \(u^*\) satisfying \(B_1u^*=z_1^*\) and \(B_2u^*=z_2^*\), such that \(\widetilde h(z^*)=h(u^*)\). For any $u$, let $z=(z_1, z_2)$ with $z_1=B_1u$ and $z_2=B_2u$. Since $z\in S$, we have, for any $u$,
\[
\begin{aligned}
f(B_1u)+c(B_2u)+h(u)\geq \widetilde f(z)+\widetilde h(z)
&\geq \widetilde f(z^*)+\widetilde h(z^*) \\
    &=f(B_1u^*)+c(B_2u^*)+h(u^*).
\end{aligned}
\]
Therefore, \(u^*\) is a global minimizer of problem~\eqref{convex_pb}.
\end{proof}

The Lagrangian function associated with problem~\eqref{convex_pb} is defined by
\begin{equation}\label{eq:Lagrangian_convex}
\mathcal L(p,q,u;\eta,\mu)
:=f(p)+c(q)+h(u)
+\langle\eta,B_1u-p\rangle
+\langle\mu,B_2u-q\rangle.
\end{equation}
Solving problem~\eqref{convex_pb} is equivalent to finding a KKT point
\(w^*=(p^*,q^*,u^*,\eta^*,\mu^*)\) satisfying
\begin{equation}\label{KKT_convex}
\begin{cases}
\mathbf{0}\in\partial f(p^*)-\eta^*,\quad \mathbf{0}\in\partial c(q^*)-\mu^*,\\[2pt]
\mathbf{0}\in\partial h(u^*)+B_1^\top\eta^*+B_2^\top\mu^*,\\[2pt]
p^*=B_1u^*,\quad q^*=B_2u^*.
\end{cases}
\end{equation}
For notational simplicity, we define the operators \(\mathcal{T}\) and \(\mathcal{M}\) as follows:
\begin{equation}\label{def_T_M}
\mathcal{T}(w) =
\begin{pmatrix}
\partial f(p) - \eta \\[2pt]
\partial c(q) - \mu \\[2pt]
\partial h(u) + B_1^{\top} \eta + B_2^{\top} \mu \\[2pt]
p - B_1 u \\[2pt]
q - B_2 u
\end{pmatrix}, \quad
\mathcal{M} =
\begin{pmatrix}
0 & 0 & 0 & 0 & 0 \\
0 & 0 & 0 & 0 & 0 \\
0 & 0 & \beta I & B_1^{\top} & B_2^{\top} \\
0 & 0 & B_1 & \frac{1}{\rho} I & 0 \\
0 & 0 & B_2 & 0 & \frac{1}{\rho} I
\end{pmatrix}.
\end{equation}
Then the KKT system in~\eqref{KKT_convex} can be compactly written as $w^*\in \mathcal{T}^{-1}(\mathbf{0})$. 

Similar to the proof of Lemma~3.1 in \cite{Sun2025Ac}, we obtain the following basic properties of the operators $\mathcal{T}$ and $\mathcal{M}$, the proofs of which are omitted. 

\begin{lemma}\label{property_T_M}
If \(\beta I-\rho(B_1^\top B_1+B_2^\top B_2)\succ0\), then \(\mathcal T\) is maximally monotone, \(\mathcal M\) is symmetric positive semidefinite with closed range, and \((\mathcal T+\mathcal M)^{-1}\) is single-valued and Lipschitz continuous.
\end{lemma}

By the definitions of $B_1$ and $B_2$, the condition
$\beta I-\rho(B_1^\top B_1+B_2^\top B_2)\succ0$
in the above lemma is equivalent to
$\beta>\rho$ and
$\beta I-\rho A^\top A-\rho\beta D^\top D/(\beta-\rho)\succ0$.

Let $\hat t=1$ if $\mathtt{acc}=0$ and $\hat t=t$ if
$\mathtt{acc}=1$. Define
\[
\widehat{\mathcal T}=(\mathcal T+\mathcal M)^{-1}\mathcal M
\quad\mbox{and}\quad
\widehat{\mathcal F}_{\hat t}
=(1-\hat t)I+\hat t\widehat{\mathcal T}.
\]
The single-valuedness of $(\mathcal T+\mathcal M)^{-1}$ implies that $w^*\in\mathcal T^{-1}(\mathbf{0})$ if and only if $w^*\in\operatorname{Fix}(\widehat{\mathcal T})$, where we denote the set of fixed points of an operator $\widehat{\mathcal T}$ by $\operatorname{Fix}(\widehat{\mathcal T})$. Inspired by the proof techniques in~\cite{Sun2025Ac}, we show that, after finitely many iterations, Algorithm~\ref{alg:prox_admm} either reduces to the fast KM iteration \cite{bot2023fast} if $\alpha>2$ and $w^{k+1}$ is updated by \eqref{ac}
for infinitely many $k$, or reduces to the Halpern iteration \cite{xu2002another} otherwise. We then establish the convergence of Algorithm~\ref{alg:prox_admm}
in what follows, for which we first provide some preliminary results.

\begin{lemma}\label{equivalent}
Under Assumption~\ref{asn_1}, let $\hat K$ be as in Lemma~\ref{unchanged}. Then, for all
$k>\hat K$, $\bar w^k=\widehat{\mathcal T}w^k$. Further, for all \(k>\hat K\), if $\alpha>2$ and $w^{k+1}$ is updated by
\eqref{ac} infinitely often $(\mathtt{acc}=1)$, then its iteration can be equivalently formulated by
\begin{equation}\label{equivalent_rep1}
\begin{array}{l}
     w^{k+1} =w^k+\dfrac{\alpha}{2(k+\alpha)}(\widehat{\mathcal F}_{\hat t} w^k-w^k)
    +\dfrac{k}{k+\alpha}(\widehat{\mathcal F}_{\hat t} w^k-\widehat{\mathcal F}_{\hat t} w^{k-1}).
\end{array}
\end{equation}
Otherwise, 
\begin{equation}\label{equivalent_rep2}
\begin{array}{l}
    w^{k+1} =\dfrac{1}{k+2}w^0+\dfrac{k+1}{k+2}\widehat{\mathcal F}_{\hat t} w^k.
\end{array}
\end{equation}
\end{lemma}
\begin{proof}
For each $k>\hat K$, Step~2 of Algorithm~\ref{alg:prox_admm} reduces to
\begin{equation}\label{adu}
\bar u^k=\operatorname{prox}_{h/\beta} \left( u^k - \frac{1}{\beta} \left[ B_1^{\top}\big(\rho(B_1 u^k - \bar{p}^k) + \bar{\eta}^k\big) + B_2^{\top}\big(\rho(B_2 u^k - \bar{q}^k) + \bar{\mu}^k\big) \right] \right).
\end{equation} 
Applying the first-order optimality conditions to \eqref{eq:p_update}, \eqref{eq:q_update} and \eqref{adu}, together with \eqref{eq:eta_update} and \eqref{eq:mu_update}, yields 
\begin{equation}\label{optimality_conditions}
\begin{cases}
\mathbf{0}\in\partial f(\bar p^k)-\bar\eta^k, 
\quad 
\mathbf{0}\in\partial c(\bar q^k)-\bar\mu^k,\\
\mathbf{0}\in\partial h(\bar u^k)+B_1^\top\bar\eta^k+B_2^\top\bar\mu^k
+\beta(\bar u^k-u^k)
+B_1^\top(\bar\eta^k-\eta^k)
+B_2^\top(\bar\mu^k-\mu^k).
\end{cases}
\end{equation}
Moreover, \eqref{eq:eta_update} and \eqref{eq:mu_update} can be rewritten as $B_1u^k+\eta^k/\rho=\bar p^k+\bar\eta^k/\rho$ and
$B_2u^k+\mu^k/\rho=\bar q^k+\bar\mu^k/\rho.$
Consequently, it gives \(\mathcal Mw^k\in(\mathcal T+\mathcal M)\bar w^k\). Since \((\mathcal T+\mathcal M)^{-1}\) is single-valued, we obtain
\(\bar w^k=(\mathcal T+\mathcal M)^{-1}\mathcal Mw^k.\) This together with Steps 3 and 4 of Algorithm~\ref{alg:prox_admm} and \cite[Remark~2.6]{Sun2025Ac} implies the conclusions hold.
\end{proof}

\begin{proposition}\label{convergence}
Under Assumptions~\ref{asn_1} and~\ref{assumption_4}, let $\hat K$ be as in Lemma~\ref{unchanged}. Then \(\bar w^k\to w^*\in\mathcal T^{-1}(\mathbf{0})\). Moreover, for all $k>\hat K$,
$\|w^k-\bar w^k\|_{\mathcal M}=O(1/k)$ if
$\alpha=2$ or $\mathtt{acc}=0$; otherwise, when
$\alpha>2$ and $\mathtt{acc}=1$,
$\|w^k-\bar w^k\|_{\mathcal M}=o(1/k)$.
\end{proposition}
\begin{proof}
For all $k>\hat K$, we first consider the cases where \(\alpha=2\) or $\mathtt{acc}=0$, which implies that $w^k$ is generated by \eqref{equivalent_rep2}. Since $\mathcal M$ is symmetric positive semidefinite by Lemma~\ref{property_T_M}, there exists a full column rank matrix $C$ such that $\mathcal M=CC^\top$. We define the shadow sequence $\xi^k=C^ \top w^k$, and introduce operators $\widetilde{T}=C^ \top (\mathcal T+\mathcal M)^{-1} C$ and $\widetilde{\mathcal{F}}_{\hat t}=(1-\hat t)I+\hat t\widetilde{T}$. Then $\widetilde{\mathcal{F}}_{\hat t}$ is nonexpansive by \cite[Proposition~2.5]{Sun2025Ac}, and $\xi^{k+1}=\xi^0/(k+2)+(k+1)\widetilde{\mathcal F}_{\hat t} \xi^k/(k+2)$. Then using \cite[Theorem~2.3]{xu2002another}, we obtain \(\xi^k\to \xi^*\in\operatorname{Fix}(\widetilde{\mathcal F}_{\hat t})\). Set $w^*=(\mathcal T+\mathcal M)^{-1} C\xi^*$. Then, the Lipschitz continuity of $(\mathcal T+\mathcal M)^{-1}$
from Lemma~\ref{property_T_M} yields
$\bar w^k=(\mathcal T+\mathcal M)^{-1}C\xi^k\to w^*$.
Moreover, $w^*\in\operatorname{Fix}(\widehat{\mathcal T})$.

Let \(v^0=\xi^0\) and \(v^{k+1}=v^0/(k+2)+(k+1)\widetilde{\mathcal F}_{\hat t} v^k/(k+2)\) for $k\geq 0$. According to \cite[Theorem~2.1]{lieder2021convergence}, we have \(\|v^k-\widetilde{\mathcal F}_{\hat t} v^k\|=O(1/k)\). For \(k>\hat K\), using the nonexpansiveness of \(\widetilde{\mathcal F}_{\hat t}\), it holds that \(\|\xi^{k+1}-v^{k+1}\|\leq (k+1)\|\xi^k-v^k\|/(k+2)\), and hence \(\|\xi^k-v^k\|\leq (\hat K+2)\|\xi^{\hat K+1}-v^{\hat K+1}\|/(k+1)=O(1/k)\). Therefore, for \(k>\hat K\), 
\(
\|\xi^k-\widetilde{\mathcal F}_{\hat t} \xi^k\|=\|\xi^k-v^k+v^k-\widetilde{\mathcal F}_{\hat t} v^k+\widetilde{\mathcal F}_{\hat t} v^k-\widetilde{\mathcal F}_{\hat t} \xi^k\|
\leq2\|\xi^k-v^k\|+\|v^k-\widetilde{\mathcal F}_{\hat t} v^k\|
=O(1/k),
\)
which implies that $\|\hat t(w^k-\bar w^k)\|_{\mathcal M}=\|w^k-\widehat{\mathcal F}_{\hat t}w^k\|_{\mathcal M}
=\|C^\top w^k-C^\top\widehat{\mathcal F}_{\hat t}w^k\|=\|\xi^k-\widetilde{\mathcal F}_{\hat t} \xi^k\|
=O(1/k)$.

If \(\alpha>2\) and $\mathtt{acc}=1$, $w^{k}$ is generated by \eqref{equivalent_rep1} for all $k> \hat K$. Lemma \ref{existence_convex} implies $\mathcal T^{-1}(\mathbf{0})\neq \emptyset$. Using sequence $\{\xi^k\}$, the conclusions can be established by an analysis similar to that used in the proof of \cite[Theorem~2.7 and Proposition~2.9]{Sun2025Ac}. 
\end{proof}

Based on the preceding proposition, we next show that  Algorithm~\ref{alg:prox_admm} converges to a weak $d$-stationary point, and establish the convergence rates of the objective function values and the residuals.

\begin{theorem}\label{th1}
Under Assumptions~\ref{asn_1} and~\ref{assumption_4}, the sequences $\{(\bar x^k,\bar y^k)\}$ and $\{(x^k,y^k)\}$
converge to a common weak $d$-stationary point of
problem~\eqref{rp1}. Let $\hat K$ and $\bar{d}$ be as in Lemma~\ref{unchanged}, and let \(w^*=(p^*,q^*,u^*,\eta^*,\mu^*)\in\mathcal T^{-1}(\mathbf{0})\) be the limit point of \(\{\bar w^k\}\). For all $k>\hat K$, if either $\alpha=2$ or
$\mathtt{acc}=0$, then
\[
\left\{
\begin{aligned}
&\left|
f(\bar p^k)+\lambda_p\|\bar q^k\|
+\lambda_0\Phi(B_3\bar u^k)
-\bigl[f(p^*)+\lambda_p\|q^*\|
+\lambda_0\Phi(B_3u^*)\bigr]
\right|
=O(1/k),\\
&\|R(\bar w^k)\|=O(1/k);
\end{aligned}
\right.
\]
otherwise, the above two estimates improve to $o(1/k)$, where
\[
R(\bar w^k):=
\begin{pmatrix}
\bar p^k-\operatorname{prox}_f(\bar p^k+\bar\eta^k)\\
\bar q^k-\operatorname{prox}_{\lambda_p\|\cdot\|}
(\bar q^k+\bar\mu^k)\\
\bar u^k-\operatorname{prox}_{(\lambda_0\Phi^{\bar d}\circ B_3)/\beta}
\big(\bar u^k-(B_1^\top\bar\eta^k+B_2^\top\bar\mu^k)/\beta\big)\\
B_1\bar u^k-\bar p^k\\
B_2\bar u^k-\bar q^k
\end{pmatrix}.
\]

\end{theorem}
\begin{proof}
It follows from Lemma~\ref{unchanged} that \(\mathcal{D}^{\bar y^k}\ni d^{k+1} =\bar d \) for all \(k>\hat K\), which implies
\(h(\bar u^k)=\lambda_0\Phi(\bar y^k)\). By Proposition~\ref{convergence}, we have $\bar w^k\to w^*$.
Since \(w^*\in\mathcal T^{-1}(\mathbf{0})\) with \(u^*=(({x^*})^\top,({y^*})^\top)^\top\), we have
\begin{equation}\label{min_F}
(x^*,y^*)\in\arg\min_{x,y} 
f(Ax)+\lambda_p\|Dx-y\|+\lambda_0\Phi^{\bar d}(y).
\end{equation}
Since $d^{k+1}=d^{\bar{y}^{k}}$ and \(\bar y^k\to y^*\), we have \(\bar d\in \mathcal{D}^{y^*}\), and hence
\(h(u^*)=\lambda_0\Phi(y^*)\). Then, the first-order optimality conditions for~\eqref{min_F} imply that \((x^*,y^*)\) is a weak \(d\)-stationary point of~\eqref{rp1}.

Using \(\bar w^k\to w^*\) again, there exists a constant \(M_{\bar w}>0\) such that \(\|\bar w^k\|\le M_{\bar w}\) for all \(k\geq 0\). We first consider the case $\alpha=2$ or $\mathtt{acc}=0$. Based on Proposition~\ref{convergence}, for all \(k>\hat K\), we have
\begin{equation}\label{rate_1}
\|w^k - \bar{w}^{k}\|_{\mathcal{M}}= O(1/k).
\end{equation}
Since the submatrix of $\mathcal{M}$ corresponding to the components $u$, $\eta$, and $\mu$ is positive definite under $\beta>\rho$ and $\beta I-\rho A^\top A-\rho\beta D^\top D/(\beta-\rho)\succ0$, \eqref{rate_1} implies that there exists a constant $C>0$ such that
\begin{equation}\label{rate_2}
\|u^k - \bar{u}^k\| + \|\eta^k - \bar{\eta}^k\| + \|\mu^k - \bar{\mu}^k\| \leq \frac{C}{k}, \quad \forall\, k > \hat K.
\end{equation}
It further follows that $(x^k,y^k) \to (x^*,y^*)$.
For brevity, let $r_1^k:=B_1u^k-\bar p^k$, $r_2^k:=B_2u^k-\bar q^k$ and $e^k:=\bar u^k-u^k$. 
Combining~\eqref{eq:eta_update}, \eqref{eq:mu_update} and \eqref{rate_2}, we have
\begin{equation}\label{ad3}
\|e^k\|\leq \frac{C}{k}\quad\mbox{and}\quad \|r_1^k\|+\|r_2^k\|\leq \frac{ C}{ \rho k},\quad k>\hat K.
\end{equation}

We first estimate the objective gap. Using \(B_1u^*=p^*\), \(B_2u^*=q^*\), and saddle-point inequality \(\mathcal L(p^*,q^*,u^*;\eta^*,\mu^*)\leq\mathcal L(\bar p^k,\bar q^k,\bar u^k;\eta^*,\mu^*)\), we obtain
\begin{equation}\label{lower_bound}
\begin{aligned}
&f(\bar p^k)+c(\bar q^k)+h(\bar u^k)
-\bigl(f(p^*)+c(q^*)+h(u^*)\bigr)\\
&\quad\geq-\langle\eta^*,r_1^k+B_1e^k\rangle
-\langle\mu^*,r_2^k+B_2e^k\rangle\\
&\quad\geq-\left[
\left(\frac1\rho+\|B_1\|\right)\|\eta^*\|
+\left(\frac1\rho+\|B_2\|\right)\|\mu^*\|
\right]\frac{C}{k}.
\end{aligned}
\end{equation}

Define
\(
\omega^k:=-B_1^\top(\bar\eta^k+\rho r_1^k)-B_2^\top(\bar\mu^k+\rho r_2^k)-\beta e^k.
\)
Combining \eqref{eq:eta_update}, \eqref{eq:mu_update} and \eqref{optimality_conditions}, we obtain
\begin{equation}\label{optimality_conditions_2}
\bar\eta^k\in\partial f(\bar p^k),\qquad
\bar\mu^k\in\partial c(\bar q^k),\qquad
\omega^k\in\partial h(\bar u^k).
\end{equation}
Applying Fenchel's equality \cite[Theorem~23.5]{RocC} to \eqref{optimality_conditions_2} and rearranging the terms, we obtain
\begin{equation}\label{upper_bound_1}
\begin{aligned}
&f^*(\bar\eta^k)+c^*(\bar\mu^k)+h^*(\omega^k)\\
&=-f(\bar p^k)-c(\bar q^k)-h(\bar u^k)
-\langle\bar\eta^k,r_1^k+B_1e^k\rangle
-\langle\bar\mu^k,r_2^k+B_2e^k\rangle\\
&\quad-\rho\langle B_1\bar u^k,r_1^k\rangle
-\rho\langle B_2\bar u^k,r_2^k\rangle
-\beta\langle\bar u^k,e^k\rangle.
\end{aligned}
\end{equation}
By $w^*\in\mathcal T^{-1}(\mathbf{0})$, we have
$p^*\in\partial f^*(\eta^*)$, $q^*\in\partial c^*(\mu^*)$ and $u^*\in\partial h^*(\omega^*)$, where $\omega^*=-B_1^\top\eta^*-B_2^\top\mu^*$. Using the convexity of the conjugate functions, it holds that
\begin{equation}\label{conjugate_lower}
\begin{aligned}
f^*(\bar\eta^k)+c^*(\bar\mu^k)+h^*(\omega^k)
&\geq f^*(\eta^*)+\langle p^*, \bar\eta^k-\eta^*\rangle+c^*(\mu^*)+\langle q^*, \bar\mu^k-\mu^* \rangle \\
&\quad+h^*(\omega^*)+\langle u^*, \omega^k-\omega^* \rangle\\
&=-\bigl(f(p^*)+c(q^*)+h(u^*)\bigr)\\
&\quad+\left\langle u^*,-\rho B_1^\top r_1^k-\rho B_2^\top r_2^k-\beta e^k\right\rangle,
\end{aligned}
\end{equation}
where the equality uses $p^*=B_1u^*$, $q^*=B_2u^*$ and $f(p^*) + c(q^*) + h(u^*)
= -\big(f^*(\eta^*) + c^*(\mu^*) + h^*(\omega^*)\big)$.
Combining~\eqref{upper_bound_1} and \eqref{conjugate_lower}, and then using $\|\bar w^k\|\leq M_{\bar w}$ and \eqref{ad3}, we obtain
\begin{equation*}
\begin{aligned}
&f(\bar p^k)+c(\bar q^k)+h(\bar u^k)
-\bigl(f(p^*)+c(q^*)+h(u^*)\bigr)\\
&\leq
-\langle\bar\eta^k,r_1^k+B_1e^k\rangle
-\langle\bar\mu^k,r_2^k+B_2e^k\rangle
-\rho\langle B_1\bar u^k,r_1^k\rangle
-\rho\langle B_2\bar u^k,r_2^k\rangle
-\beta\langle\bar u^k,e^k\rangle\\
&\quad+\left\langle
u^*,\rho B_1^\top r_1^k+\rho B_2^\top r_2^k+\beta e^k
\right\rangle\\
&\leq
\left[
M_{\bar w}\left(
\frac{1}{\rho}
+2(\|B_1\|+\|B_2\|)
+\beta
\right)
+\|u^*\|\bigl(
\|B_1\|+\|B_2\|+\beta
\bigr)
\right]\frac{C}{k},
\qquad k>\hat K.
\end{aligned}
\end{equation*}
This together with~\eqref{lower_bound} yields the $O(1/k)$ objective-gap estimate.

We next establish the $O(1/k)$ convergence rate for the KKT residual. By~\eqref{optimality_conditions_2}, it holds that
\[
\bar p^k-\operatorname{prox}_f(\bar p^k+\bar\eta^k)=\mathbf{0},
\qquad
\bar q^k-\operatorname{prox}_{c}(\bar q^k+\bar\mu^k)=\mathbf{0}.
\]
The two primal feasibility components satisfy
\begin{equation}\label{primal_residual_bounds}
\left\{
\begin{aligned}
\|B_1\bar u^k-\bar p^k\|&\leq\|r_1^k+B_1e^k\|\leq\left(\frac1\rho+\|B_1\|\right)\frac{C}{k},\\
\|B_2\bar u^k-\bar q^k\|&\leq \|r_2^k+B_2e^k\|\leq\left(\frac1\rho+\|B_2\|\right)\frac{C}{k},
\end{aligned}
\qquad k>\hat K.\right.
\end{equation}
From \eqref{adu} and the nonexpansiveness of the proximal operator $\operatorname{prox}_{h/\beta}$, we have
\begin{equation}\label{u_res_estimate}
\begin{aligned}
&\left\| \bar{u}^k - \operatorname{prox}_{h/\beta} \left( \bar{u}^k - \frac{1}{\beta}(B_1^{\top} \bar{\eta}^k + B_2^{\top} \bar{\mu}^k) \right) \right\| \\
&\leq \left\| u^k- \bar{u}^k-\frac{1}{\beta}\rho B_1^{\top}r_1^k-\frac{1}{\beta}\rho B_2^{\top}r_2^k\right\|\\
&\leq\|e^k\|+\frac{\rho}{\beta}\bigl(\|B_1\|\|r_1^k\|+\|B_2\|\|r_2^k\|\bigr)
\leq\left(1+\frac{\|B_1\|+\|B_2\|}{\beta}\right)\frac{C}{k},\qquad k>\hat K.
\end{aligned}
\end{equation}
Thus, $\|R(\bar w^k)\|=O(1/k)$. If $\alpha>2$ and $\mathtt{acc}=1$, we have $\|w^k-\bar w^{k}\|_{\mathcal M}=o(1/k)$ from Proposition~\ref{convergence}. The same estimates as above, with $o(1/k)$ replacing $ O(1/k)$, yield the improved rates.
\end{proof}

\begin{remark}\label{remark:convex_extension}
The full row rank condition on $D$ is not required in
Theorem~\ref{th1}. If $D$ has full row rank, then the limit point
of $\{\bar x^k\}$ (or $\{x^k\}$) is also a
$\nu$-strong local minimizer of problem \eqref{ob0}.
When directly applied to solving the convex problem~\eqref{convex_pb}, Algorithm~\ref{alg:prox_admm} can be simplified without requiring the index identification and safeguarding steps, and reduces to the fixed-point iteration~\eqref{equivalent_rep1}. In this case, the KKT residual and the objective function value gap converge at rates \( O(1/k)\) for \(\alpha=2\) and \(o(1/k)\) for \(\alpha>2\). The above convergence analysis further suggests that this simplified fixed-point framework can be extended to the more general convex composite problem 
\(\min_x\sum_{i=1}^{N}f_i(B_i x)+h(x)\), where \(f_i\) and \(h\) are proper closed convex functions with computable proximal mappings, and each $B_i \in \mathbb{R}^{r \times n}$, through the corresponding construction of \(\mathcal T\) and \(\mathcal M\).
\end{remark}

\section{Numerical experiments}\label{sec:experiments}

We conduct two numerical experiments to evaluate
Algorithm~\ref{alg:prox_admm}. The first concerns one-dimensional
piecewise-constant signal denoising, for which the first-order
difference matrix has full row rank.
The second concerns total variation denoising on the MNIST dataset,
for which the two-dimensional difference matrix does not have full row rank. 
All experiments were run in Python on a Windows~11 PC with a
13th-generation Intel Core i5 processor and 16~GB of RAM. The parameter settings for the compared algorithms are provided in the publicly available code (see the Code availability statement).

\subsection{Piecewise-constant signal denoising}\label{Signal}

We recover a piecewise-constant signal \(b\in\mathbb{R}^n\)
from a noisy observation \(\hat b\in\mathbb{R}^n\).
Table~\ref{tab:models_algorithms} lists the tested models and corresponding
solvers. Let \(D\in\mathbb{R}^{(n-1)\times n}\) be the first-order difference
matrix with \(-1\) on the main diagonal and \(1\) on the
superdiagonal.

\begin{table}[h]
\centering
\footnotesize
\caption{Optimization models and corresponding solvers}
\label{tab:models_algorithms}
\begin{tabular}{@{} l c c @{}}
\toprule
\textbf{Model} & \textbf{Problem formulation} & \textbf{Solver} \\
\midrule
\(\ell_1\)-\(C_{\ell_1}\)
&
$
\min_{x,y}\;
\|x-\hat b\|_1/n
+\lambda_p\|Dx-y\|
+\lambda_0\Phi(y)
$
&
SAP-ADMM\textsuperscript{H}/SAP-ADMM
\\
\midrule
\(\ell_1\)-\(\ell_{1/2}\)
&
$
\min_x\;
\|x-\hat b\|_1/n
+\lambda_{\ell_{1/2}}\|Dx\|_{1/2}^{1/2}
$
&
SDCAM\textsuperscript{1}~\cite{L2019A}
\\
\midrule
\(\ell_1\)-\(\ell_1\)
&
$
\min_{x,y}\;
\|x-\hat b\|_1/n
+\lambda_p\|Dx-y\|
+\lambda_{\ell_1}\|y\|_1
$
&
SAP-ADMM\textsuperscript{1}
\\
\midrule
\(\ell_2\)-\(\ell_{1/2}\)
&
$
\min_x\;
\|x-\hat b\|^2/2
+\lambda_{\ell_{1/2}}\|Dx\|_{1/2}^{1/2}
$
&
SDCAM\textsuperscript{2}~\cite{L2019A}
\\
\midrule
\(\ell_2\)-\(\ell_0\)
&
$
\min_x\;
\|x-\hat b\|^2/2
+\lambda_{\ell_0}\|Dx\|_0
$
&
pADMM~\cite{bot2020The}
\\
\bottomrule
\end{tabular}

\vspace{2pt}
\begin{minipage}{0.98\linewidth}
\footnotesize
\textbf{Note.}
The abbreviation \(\ell_1\)-\(C_{\ell_1}\) denotes the
\(\ell_1\)-capped-\(\ell_1\) model. Motivated by the acceleration observed when incorporating a restart
strategy into SAP-ADMM with $\alpha=2$ and $t=1$, we consider
SAP-ADMM\textsuperscript{H}, the Halpern variant with $\alpha=2$, $t=1$,
and a single restart at iteration $1000$. In comparison, SAP-ADMM
uses $\alpha=15$ and $t=1.5$. SDCAM\textsuperscript{1} and SDCAM\textsuperscript{2}
correspond to the $\ell_1$-$\ell_{1/2}$ and
$\ell_2$-$\ell_{1/2}$ models, respectively.
\end{minipage}
\end{table}

We set \(n=1000\). The segment lengths and signal levels are sampled from \([50,150]\) and \([-5,10]\), respectively, with adjacent levels differing by more than \(2\). The observation \(\hat b\) is generated by adding Gaussian noise with standard deviation \(0.5\) and impulse noise with probability \(\pi \in \{0,0.05,\ldots,0.20\}\). Each corrupted entry is perturbed uniformly on \([-5,5]\). We perform $50$ independent trials for each $\pi$.
Jump-location recovery is measured by the $F_1$ score, which is computed as
\[
F_1
=
\frac{
2|\mathcal{S}_{\mathrm{true}}\cap\mathcal{S}_{\mathrm{est}}|
}{
2|\mathcal{S}_{\mathrm{true}}\cap\mathcal{S}_{\mathrm{est}}|
+
|\mathcal{S}_{\mathrm{est}}\setminus\mathcal{S}_{\mathrm{true}}|
+
|\mathcal{S}_{\mathrm{true}}\setminus\mathcal{S}_{\mathrm{est}}|
}.
\]
Clearly, $F_1=1$ corresponds to exact support recovery.
Set
\(\varepsilon_{\mathrm{supp}}=10^{-6}\), and define the set of true jump locations \(\mathcal{S}_{\mathrm{true}}
=\bigl\{i:|(Db)_i|>\varepsilon_{\mathrm{supp}}\bigr\}\). For SAP-ADMM, SAP-ADMM\textsuperscript{H}, and
SAP-ADMM\textsuperscript{1}, the estimated jump locations are
\(\mathcal{S}_{\mathrm{est}}
=\{i:|\bar y_i|>\varepsilon_{\mathrm{supp}}\}\),
where \(\bar y\) is returned by the corresponding algorithm
at termination.
For the remaining algorithms, \(\mathcal{S}_{\mathrm{est}}\)
consists of the indices for which the absolute value of the
corresponding output component exceeds
\(\varepsilon_{\mathrm{supp}}\).

For the \(\ell_1\)-\(C_{\ell_1}\) model, we set
\(\lambda_0=0.016\) and \(\lambda_p=500L_f\), where
\(L_f=1/\sqrt{n}\), ensuring that Assumption~\ref{asp} holds.
For SAP-ADMM and SAP-ADMM\textsuperscript{H}, we set
\(N_{\mathrm{max}}=500\), \(\rho=2/n\), and
\(\beta=6\rho+10^{-8}\).
The initial points are chosen as \(x^0=\hat b\),
\(y^0=\mathbf{0}\), \(p^0=x^0\), \(q^0=Dx^0-y^0\),
\(\eta^0=\mathbf{0}\), and \(\mu^0=\mathbf{0}\).
To improve support identification, we set \(\nu_0=2\) and update
\[
\nu_k
=\max\left\{
\nu_{k-1}\gamma_\nu,\;
0.99\,\lambda_0/(3\lambda_p)
\right\},
\qquad k\geq1,
\]
where \(\gamma_\nu=0.999\) for \(k\leq1500\) and
\(\gamma_\nu=0.95\) otherwise. Here, \(k\) counts effective
iterations, excluding rejected acceleration trials.
After finitely many iterations, \(\nu_k\) reaches and remains
at \(\nu_{\min}=0.99\,\lambda_0/(3\lambda_p)\), which satisfies
Assumption~\ref{asn_1}.
Define
\[
\Delta_k
:=
\max\bigl\{
\|\bar p^k-p^k\|,
\|\bar q^k-q^k\|,
\|\bar x^k-x^k\|,
\|\bar y^k-y^k\|,
\|\bar\eta^k-\eta^k\|,
\|\bar\mu^k-\mu^k\|
\bigr\}.
\]

Note that $\Delta_k=0$ implies $R(\bar w^k)=\mathbf{0}$.
We terminate SAP-ADMM and SAP-ADMM\textsuperscript{H} when
\(\Delta_k<2\times10^{-4}\), provided that \(\nu_k\) has reached
\(\nu_{\min}\) and at least \(50\) additional iterations
have been performed, or when the iteration count reaches
\(2\times10^4\).
 
\begin{figure}[htbp]
    \centering
    \includegraphics[width=\textwidth]{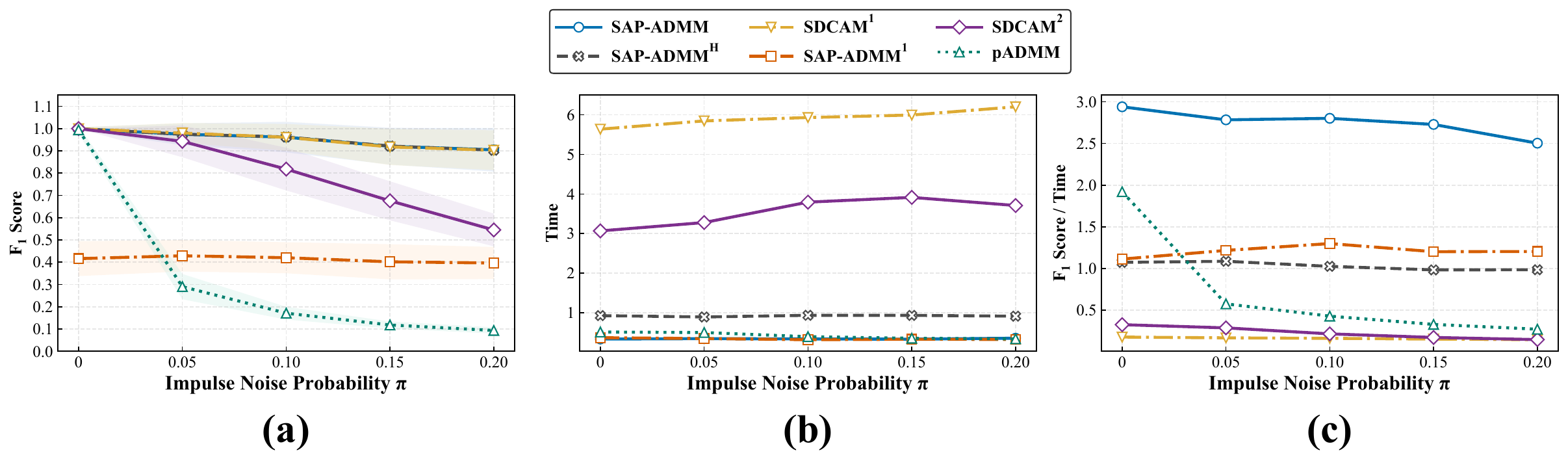}
    \caption{Performance comparison of six algorithms under
    different impulse noise probabilities:
    (a) average \(F_1\) score with standard deviation;
    (b) average computation time;
    (c) efficiency ratio.}
    \label{fig:f1_time_efficiency}
\end{figure}

Fig.~\ref{fig:f1_time_efficiency} reports the average \(F_1\)
score and computation time, as well as the efficiency ratio, defined
as the average \(F_1\) score divided by the average computation time.
The results in Fig.~\ref{fig:f1_time_efficiency}(a) show that
SAP-ADMM and SAP-ADMM\textsuperscript{H} applied to the
\(\ell_1\)-\(C_{\ell_1}\) model, and SDCAM\textsuperscript{1}
applied to the \(\ell_1\)-\(\ell_{1/2}\) model, consistently
attain average \(F_1\) scores above \(0.8\).
In contrast, SAP-ADMM\textsuperscript{1} applied to the
\(\ell_1\)-\(\ell_1\) model is less accurate in identifying jumps.
This comparison illustrates the advantage of nonconvex penalties
for sparse jump recovery.
SDCAM\textsuperscript{2} and pADMM, which solve the
\(\ell_2\)-\(\ell_{1/2}\) and \(\ell_2\)-\(\ell_0\) models,
respectively, perform well without impulse noise but deteriorate
rapidly as \(\pi\) increases, compared with the algorithms
solving models with the \(\ell_1\) loss.
The computation times in Fig.~\ref{fig:f1_time_efficiency}(b) indicate
that pADMM, SAP-ADMM, SAP-ADMM\textsuperscript{H}, and
SAP-ADMM\textsuperscript{1} incur relatively low computational
costs when solving their respective models. Fig.~\ref{fig:f1_time_efficiency}(c) shows that SAP-ADMM
outperforms SAP-ADMM\textsuperscript{H} and all the other
compared algorithms when applied to their respective models in terms of computational efficiency,
indicating the advantage of the safeguarded acceleration strategy.
Specifically, SAP-ADMM achieves a higher efficiency ratio than
SAP-ADMM\textsuperscript{1}, which in turn outperforms
SAP-ADMM\textsuperscript{H}.

\begin{table}[htbp]
\centering
\footnotesize
\caption{Average MSE and computation time for piecewise-constant signal denoising}
\label{tab:mse_time_comparison}
\renewcommand{\arraystretch}{1.10}
\setlength{\tabcolsep}{3.5pt}
\begin{tabular}{@{}lccccc@{}}
\toprule
\multirow{2}{*}{Algorithm}
& \multicolumn{5}{c}{Impulse noise probability \(\pi\)} \\
\cmidrule(lr){2-6}
& \(0.00\) & \(0.05\) & \(0.10\) & \(0.15\) & \(0.20\) \\
\midrule
SAP-ADMM
& \(0.0059/\mathbf{0.3399}\)
& \(\mathbf{0.0114}/\mathbf{0.3502}\)
& \(0.0189/0.3432\)
& \(\mathbf{0.0308}/0.3375\)
& \(0.0374/0.3614\) \\
SAP-ADMM\textsuperscript{H}
& \(0.0057/0.9310\)
& \(0.0115/0.8970\)
& \(\mathbf{0.0187}/0.9386\)
& \(0.0309/0.9373\)
& \(\mathbf{0.0369}/0.9171\) \\
SDCAM\textsuperscript{1}
& \(0.0157/5.6352\)
& \(0.0219/5.8439\)
& \(0.0287/5.9286\)
& \(0.0397/5.9919\)
& \(0.0449/6.1992\) \\
SAP-ADMM\textsuperscript{1}
& \(0.2256/0.3740\)
& \(0.2769/0.3527\)
& \(0.2230/\mathbf{0.3236}\)
& \(0.2075/\mathbf{0.3345}\)
& \(0.2019/\mathbf{0.3295}\) \\
SDCAM\textsuperscript{2}
& \(\mathbf{0.0033}/3.0686\)
& \(0.0174/3.2793\)
& \(0.0423/3.7940\)
& \(0.0810/3.9140\)
& \(0.1394/3.7091\) \\
pADMM
& \(0.0036/0.5178\)
& \(0.3933/0.5048\)
& \(0.7680/0.3998\)
& \(1.1339/0.3598\)
& \(1.5655/0.3453\) \\
\bottomrule
\end{tabular}

\vspace{2pt}
\begin{minipage}{0.96\linewidth}
\footnotesize
\textbf{Note.}
Each entry is reported as MSE/computation time (s).
\end{minipage}
\end{table}

Table~\ref{tab:mse_time_comparison} reports the average mean
squared error (MSE) and computation time.
The MSE is computed as \(\mathrm{MSE}=\|\bar x-b\|^2/n\),
where \(\bar x\) is the restored signal returned by the
corresponding algorithm at termination.
When \(\pi=0\), the \(\ell_2\)-\(\ell_{1/2}\) model
solved by SDCAM\textsuperscript{2} yields the lowest MSE.
For \(\pi\geq0.05\), the lowest MSE values are
obtained from the \(\ell_1\)-\(C_{\ell_1}\) model using
SAP-ADMM or SAP-ADMM\textsuperscript{H}.
Moreover, the \(\ell_1\)-\(C_{\ell_1}\) model solved by
SAP-ADMM and SAP-ADMM\textsuperscript{H}, together with the
\(\ell_1\)-\(\ell_{1/2}\) model solved by
SDCAM\textsuperscript{1}, consistently yields relatively low
MSE values across noise levels, illustrating the effectiveness
of the \(\ell_1\) loss combined with nonconvex penalties.
For the \(\ell_1\)-\(C_{\ell_1}\) model, SAP-ADMM substantially
reduces the computation time relative to SAP-ADMM\textsuperscript{H}.
Hence, we focus on SAP-ADMM
in all subsequent experiments.

\begin{figure}[htbp]
    \centering
    \includegraphics[width=\textwidth]{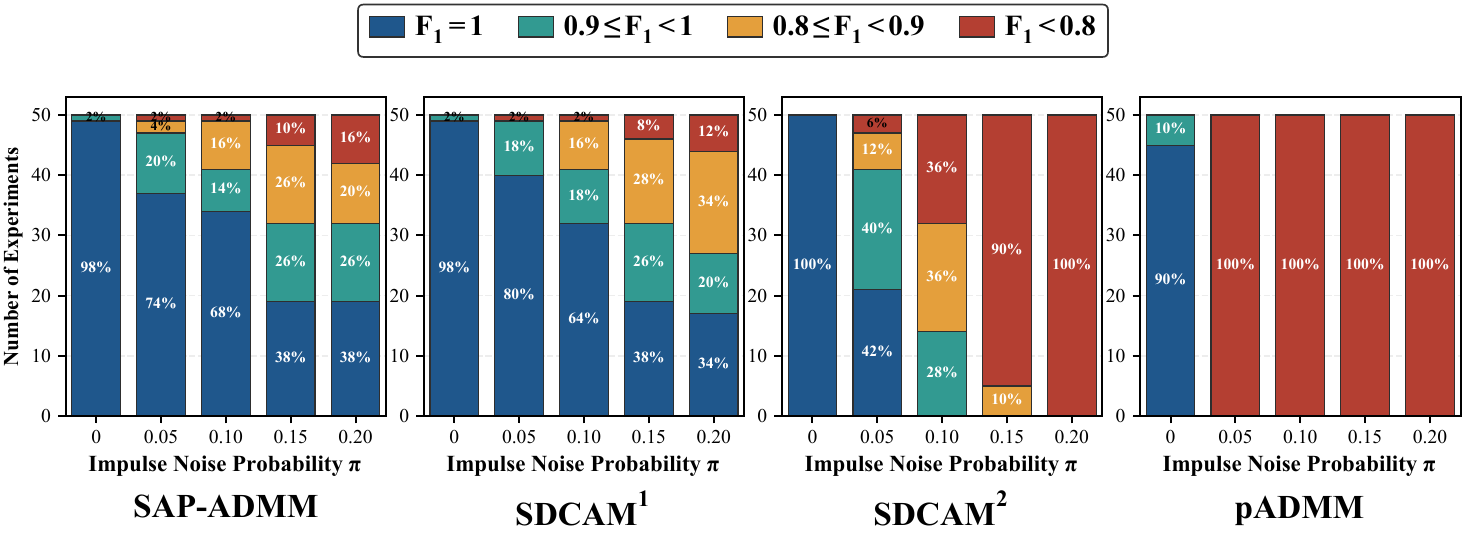}
    \caption{Distribution of \(F_1\) scores under different
    impulse noise probabilities.}
    \label{fig:f1_distribution}
\end{figure}

Fig.~\ref{fig:f1_distribution} presents the distributions
of \(F_1\) scores over \(50\) trials.
SAP-ADMM\textsuperscript{1} applied to the
\(\ell_1\)-\(\ell_1\) model is omitted because it rarely
recovers the exact support.
SAP-ADMM and SDCAM\textsuperscript{1}, applied to the
\(\ell_1\)-\(C_{\ell_1}\) and
\(\ell_1\)-\(\ell_{1/2}\) models, respectively, attain
exact recovery in more than \(70\%\) of the trials for
\(\pi\leq0.05\) and in more than \(30\%\) of the trials
for \(\pi\in\{0.15,0.20\}\).
SDCAM\textsuperscript{2} and pADMM, applied to the
\(\ell_2\)-\(\ell_{1/2}\) and \(\ell_2\)-\(\ell_0\)
models, respectively, attain reliable exact recovery when
\(\pi=0\) but do not recover the exact support in any trial when
\(\pi\in\{0.15,0.20\}\).
For a direct comparison, Fig.~\ref{fig:recovery} shows the
reconstructions obtained in one trial for \(\pi=0\) and \(0.20\),
using \(N_{\mathrm{max}}=2000\) for SAP-ADMM.
When \(\pi=0\), all compared algorithms except
SAP-ADMM\textsuperscript{1} closely recover the clean signal
when solving their respective models.
When \(\pi=0.20\), the \(\ell_1\)-\(C_{\ell_1}\) model
solved by SAP-ADMM and the \(\ell_1\)-\(\ell_{1/2}\) model
solved by SDCAM\textsuperscript{1} yield the most accurate
reconstructions, whereas those obtained from the
\(\ell_2\)-loss models using SDCAM\textsuperscript{2}
and pADMM retain noticeable noise.

\begin{figure}[htbp]
    \centering
    \includegraphics[width=\textwidth]{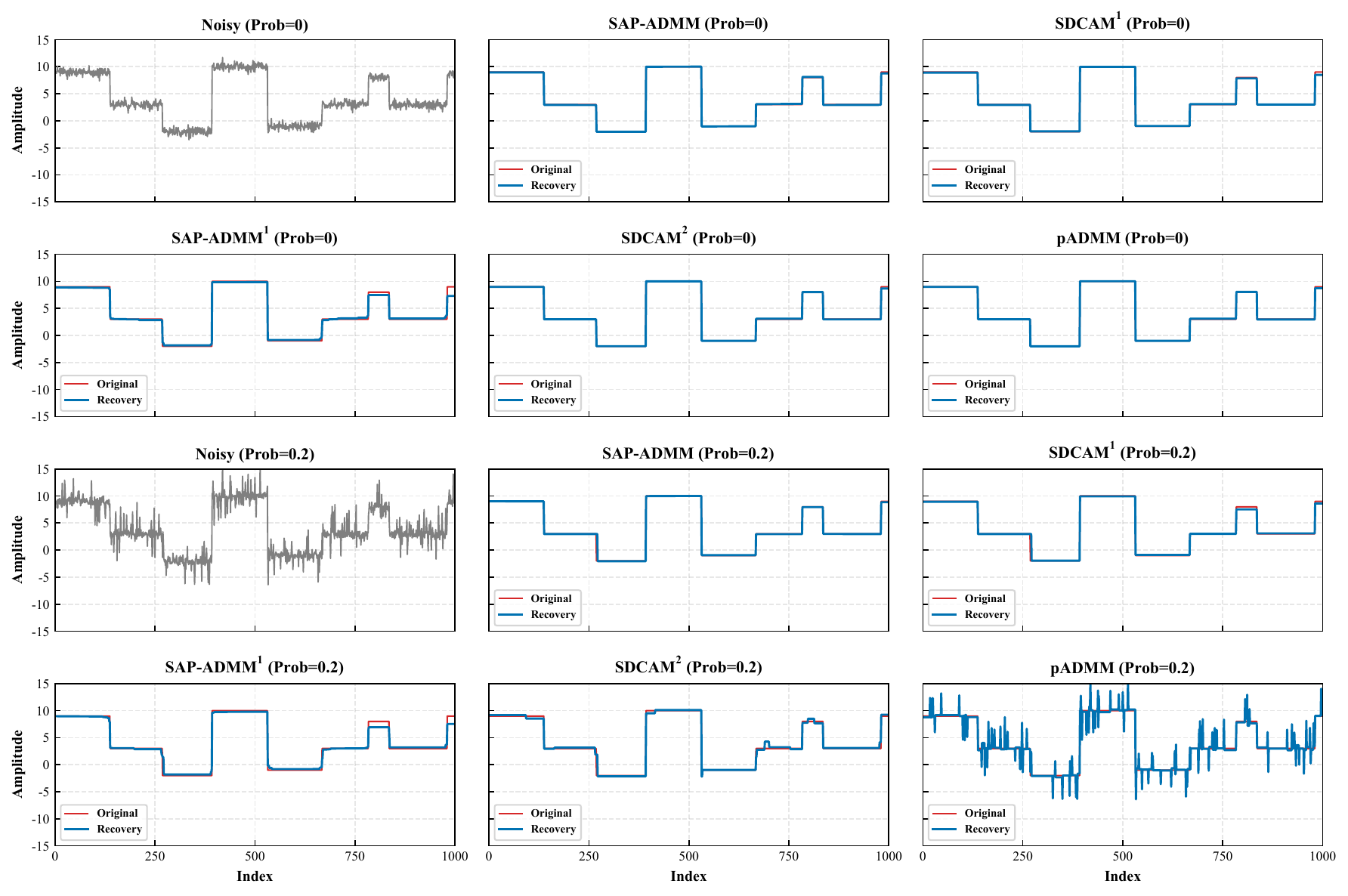}
    \caption{Signal recovery by the compared algorithms applied
    to their respective models under impulse noise probabilities
    of \(0\%\) and \(20\%\).}
    \label{fig:recovery}
\end{figure}

We further evaluate SAP-ADMM on the problem
\(
\min_{x,y}\;
\|Ax-\hat b\|_1/m
+\lambda_p\|Dx-y\|
+\lambda_0\Phi(y),
\)
with \(m=800\) and \(n=1000\).
In each trial, \(A\in\mathbb{R}^{m\times n}\) is generated
with independent standard Gaussian entries and normalized
so that \(\|A\|=1\).
The signal \(b\) and matrix \(D\) are constructed as before.
The observation \(\hat b\) is generated from \(Ab\), with Gaussian standard deviation
\(0.2\) and impulse noise probability \(0.05\).
We set \(\lambda_0=0.0018\), \(\lambda_p=500/\sqrt{m}\),
and \(\rho=0.6/n\), and initialize
\(x^0=A^\top\hat b\) and \(p^0=Ax^0\). Here, \(\nu\) is updated as above, except that
\(\gamma_\nu=0.9995\) for \(k\leq4000\) and
\(\gamma_\nu=0.95\) thereafter.
The remaining parameters, initial points, and stopping
criterion for SAP-ADMM are unchanged.
We perform \(20\) independent trials.
Within each trial, the same problem instance is used for all
\(N_{\mathrm{max}}\in\{200,300,400,500,600\}\).

\begin{table}[htbp]
\centering
\small
\caption{Average results of SAP-ADMM for different values of \(N_{\mathrm{max}}\)}
\label{tab:nmax_general_A}
\renewcommand{\arraystretch}{1.05}
\setlength{\tabcolsep}{10pt}
\begin{tabular}{cccc}
\toprule
\(N_{\mathrm{max}}\)
& Average \(F_1\)
& Average iterations
& Average MSE \\
\midrule
200 & 1.0000 & 8239.95 & 0.0034 \\
300 & 1.0000 & 8239.95 & 0.0034 \\
400 & 1.0000 & 8239.95 & 0.0034 \\
500 & 1.0000 & 8239.95 & 0.0034 \\
600 & 1.0000 & 8239.95 & 0.0034 \\
\bottomrule
\end{tabular}
\end{table}

Table~\ref{tab:nmax_general_A} reports the average \(F_1\)
score, iteration count, and MSE of the SAP-ADMM algorithm over the \(20\) trials.
The average \(F_1\) score and MSE demonstrate the effectiveness
of SAP-ADMM in signal recovery when \(A\) is not an identity matrix.
The reported averages remain unchanged across all tested
values of \(N_{\mathrm{max}}\), indicating that SAP-ADMM
is relatively insensitive to this parameter under the
considered settings.

All these results demonstrate that SAP-ADMM achieves a favorable balance
between reconstruction accuracy and computational efficiency
for piecewise-constant signal denoising.

\subsection{MNIST total variation denoising}

We consider total variation denoising of \(28\times28\)
MNIST images~\cite{LeCun1998}.
Let \(b\in\mathbb{R}^n\), with \(n=n_1n_2\), represent the
columnwise vectorization of a clean \(n_1\times n_2\) image,
and let \(\hat b\) be its noisy observation.
Here, \(D=[D_h^\top,D_v^\top]^\top\), with
\(D_h=L_{n_2}\otimes I_{n_1}\) and
\(D_v=I_{n_2}\otimes L_{n_1}\), where \(\otimes\) denotes the
Kronecker product and \(L_{n_i}\in\mathbb{R}^{n_i\times n_i}\)
is the first-order difference matrix with entries
\((L_{n_i})_{j,j}=-1\) and \((L_{n_i})_{j,j+1}=1\)
for \(j=1,\ldots,n_i-1\), and all other entries equal to zero.
By \(D\in\mathbb{R}^{2n_1n_2\times n_1n_2}\) and
\(\operatorname{rank}(D)\leq n_1n_2<2n_1n_2\), it follows
that \(D\) does not have full row rank.
Since the convergence of pADMM requires \(D\) to
have full row rank, we test all models in Table~\ref{tab:models_algorithms} except the $\ell_2$-$\ell_0$ model solved by pADMM.

For each digit \(0,\ldots,9\), we select one image and generate
\(20\) noisy observations by replacing \(10\%\) of its pixels
with independent samples from the uniform distribution on
\([0,1]\), adding Gaussian noise with standard deviation
\(0.2\), and clipping the resulting values to \([0,1]\).
For the \(\ell_1\)-\(C_{\ell_1}\) model, we set
\(\lambda_0=8\times10^{-5}\) and \(\lambda_p=2L_f\).
For SAP-ADMM, we set \(\rho=0.5/n\) and
\(\beta=10\rho+10^{-8}\), with \(n=28\times28\).
Starting from \(\nu_0=0.1\), we update \(\nu_k\) according to
\[
\nu_k
=
\max\left\{
\nu_0\,0.99^k,\;
0.999\,\lambda_0/(3\lambda_p)
\right\}.
\]
Here, \(\nu_{\min}=0.999\,\lambda_0/(3\lambda_p)\).
The stopping criterion is \(\Delta_k<4\times10^{-4}\), subject
to the same requirement of at least \(50\) additional effective
iterations after \(\nu_k\) reaches \(\nu_{\min}\).
The remaining parameters, initial points, and maximum iteration
number are the same as those used in \cref{Signal}.

We evaluate restoration quality using the peak signal-to-noise
ratio (PSNR) for pixelwise reconstruction accuracy, the structural
similarity index measure (SSIM)~\cite{Wang2004} for structural
similarity, and the gradient magnitude similarity deviation
(GMSD)~\cite{Xue2014} for spatial variation in local gradient
magnitude similarity.
Higher PSNR and SSIM values and lower GMSD values indicate
better restoration quality.

Table~\ref{tab:mnist_four_models_compact} reports the average
PSNR, computation time, SSIM, and GMSD.
The \(\ell_1\)-\(C_{\ell_1}\) model solved by SAP-ADMM
yields the best average PSNR, SSIM, and GMSD for every digit.
Compared with the \(\ell_1\)-\(\ell_1\) model solved by
SAP-ADMM\textsuperscript{1}, it improves average PSNR by
approximately \(3.22\%\), increases SSIM by \(14.45\%\), and
reduces GMSD by \(27.40\%\).
When solving their respective models, SAP-ADMM requires,
on average, more computation time than SAP-ADMM\textsuperscript{1},
but less than SDCAM\textsuperscript{1} and
SDCAM\textsuperscript{2}.

\begin{table}[htbp]
\centering
\caption{Comparison of four algorithms applied to their respective
models for MNIST image denoising}
\label{tab:mnist_four_models_compact}
\scriptsize
\setlength{\tabcolsep}{3.6pt}
\renewcommand{\arraystretch}{1.08}
\begin{tabular}{@{}c l c c c c@{}}
\toprule
Digit & Metric
& SAP-ADMM
& SDCAM\textsuperscript{1}
& SAP-ADMM\textsuperscript{1}
& SDCAM\textsuperscript{2} \\
\midrule
\multirow{2}{*}{0}
& \(P\uparrow/T\downarrow\) & \(\mathbf{18.05}/0.1833\) & \(17.44/3.4761\) & \(17.77/\mathbf{0.0674}\) & \(15.02/1.9773\) \\
& \(S\uparrow/G\downarrow\) & \(\mathbf{0.629}/\mathbf{0.1958}\) & \(0.542/0.2130\) & \(0.561/0.2555\) & \(0.392/0.2457\) \\
\midrule
\multirow{2}{*}{1}
& \(P\uparrow/T\downarrow\) & \(\mathbf{21.25}/0.3365\) & \(19.98/2.6230\) & \(20.48/\mathbf{0.0743}\) & \(16.39/1.8851\) \\
& \(S\uparrow/G\downarrow\) & \(\mathbf{0.484}/\mathbf{0.1491}\) & \(0.372/0.2021\) & \(0.388/0.2781\) & \(0.222/0.2170\) \\
\midrule
\multirow{2}{*}{2}
& \(P\uparrow/T\downarrow\) & \(\mathbf{17.57}/0.1746\) & \(16.89/4.2705\) & \(17.25/\mathbf{0.0656}\) & \(14.68/2.8998\) \\
& \(S\uparrow/G\downarrow\) & \(\mathbf{0.609}/\mathbf{0.2286}\) & \(0.529/0.2351\) & \(0.560/0.2621\) & \(0.397/0.2757\) \\
\midrule
\multirow{2}{*}{3}
& \(P\uparrow/T\downarrow\) & \(\mathbf{18.92}/0.4264\) & \(18.02/4.2294\) & \(18.52/\mathbf{0.0736}\) & \(15.28/2.0205\) \\
& \(S\uparrow/G\downarrow\) & \(\mathbf{0.557}/\mathbf{0.1941}\) & \(0.482/0.2143\) & \(0.497/0.2511\) & \(0.328/0.2573\) \\
\midrule
\multirow{2}{*}{4}
& \(P\uparrow/T\downarrow\) & \(\mathbf{18.92}/0.2767\) & \(18.57/2.8153\) & \(18.12/\mathbf{0.0704}\) & \(15.61/1.7116\) \\
& \(S\uparrow/G\downarrow\) & \(\mathbf{0.435}/\mathbf{0.1981}\) & \(0.383/0.2102\) & \(0.391/0.2728\) & \(0.253/0.2410\) \\
\midrule
\multirow{2}{*}{5}
& \(P\uparrow/T\downarrow\) & \(\mathbf{18.31}/0.4124\) & \(17.47/3.6850\) & \(17.57/\mathbf{0.0681}\) & \(15.19/2.2182\) \\
& \(S\uparrow/G\downarrow\) & \(\mathbf{0.550}/\mathbf{0.2063}\) & \(0.466/0.2192\) & \(0.473/0.2602\) & \(0.327/0.2453\) \\
\midrule
\multirow{2}{*}{6}
& \(P\uparrow/T\downarrow\) & \(\mathbf{19.44}/0.1490\) & \(18.38/3.5776\) & \(18.79/\mathbf{0.0716}\) & \(15.48/1.9762\) \\
& \(S\uparrow/G\downarrow\) & \(\mathbf{0.521}/\mathbf{0.1761}\) & \(0.432/0.2069\) & \(0.458/0.2705\) & \(0.305/0.2314\) \\
\midrule
\multirow{2}{*}{7}
& \(P\uparrow/T\downarrow\) & \(\mathbf{19.19}/0.2851\) & \(18.14/4.4902\) & \(18.35/\mathbf{0.0775}\) & \(15.56/2.4998\) \\
& \(S\uparrow/G\downarrow\) & \(\mathbf{0.531}/\mathbf{0.1824}\) & \(0.424/0.2055\) & \(0.440/0.2654\) & \(0.289/0.2271\) \\
\midrule
\multirow{2}{*}{8}
& \(P\uparrow/T\downarrow\) & \(\mathbf{18.15}/0.3959\) & \(17.26/3.5041\) & \(17.41/\mathbf{0.0677}\) & \(14.66/1.9542\) \\
& \(S\uparrow/G\downarrow\) & \(\mathbf{0.531}/\mathbf{0.1994}\) & \(0.455/0.2204\) & \(0.470/0.2659\) & \(0.323/0.2602\) \\
\midrule
\multirow{2}{*}{9}
& \(P\uparrow/T\downarrow\) & \(\mathbf{18.93}/0.5285\) & \(18.09/4.1470\) & \(18.57/\mathbf{0.0707}\) & \(15.38/2.0600\) \\
& \(S\uparrow/G\downarrow\) & \(\mathbf{0.533}/\mathbf{0.1897}\) & \(0.441/0.2087\) & \(0.463/0.2622\) & \(0.299/0.2404\) \\
\midrule
\multirow{2}{*}{Avg.}
& \(P\uparrow/T\downarrow\) & \(\mathbf{18.87}/0.3168\) & \(18.02/3.6818\) & \(18.28/\mathbf{0.0707}\) & \(15.33/2.1203\) \\
& \(S\uparrow/G\downarrow\) & \(\mathbf{0.538}/\mathbf{0.1920}\) & \(0.453/0.2135\) & \(0.470/0.2644\) & \(0.314/0.2441\) \\
\bottomrule
\end{tabular}

\vspace{2pt}
\begin{minipage}{\linewidth}
\footnotesize
\textbf{Note.}
\(P\), \(T\), \(S\), and \(G\) denote PSNR, computation time (s),
SSIM, and GMSD, respectively.
The symbols \(\uparrow\) and \(\downarrow\) indicate that
larger and smaller values are preferable, respectively.
\end{minipage}
\end{table}

\begin{figure}[h]
    \centering
    \includegraphics[width=0.7\textwidth]{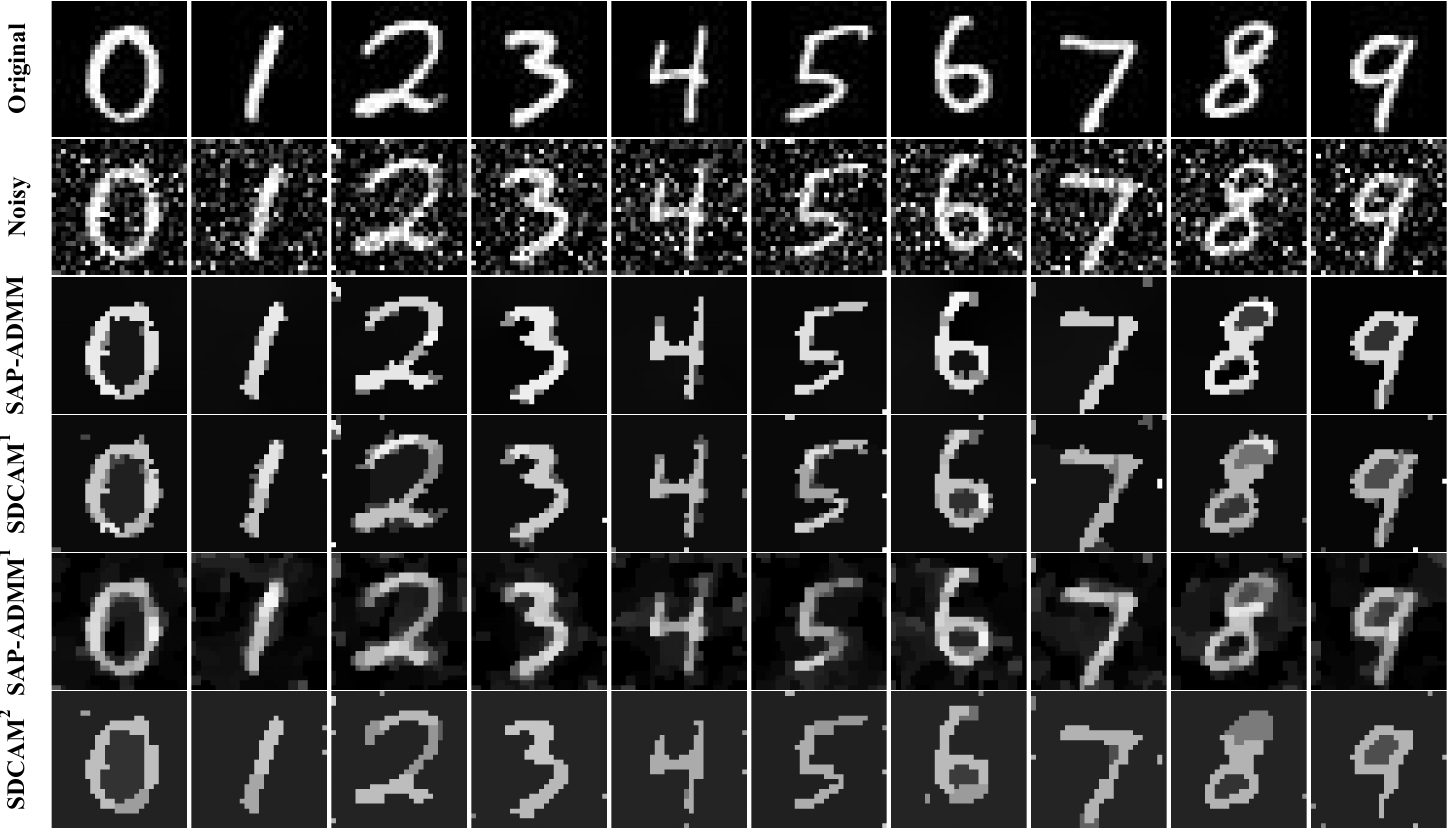}
    \caption{Original images, noisy observations, and restorations
    obtained by the four compared algorithms applied to their
    respective models.}
    \label{fig:mnist}
\end{figure}

For a visual comparison, Fig.~\ref{fig:mnist} presents the
restored images from one trial.
The \(\ell_1\)-\(C_{\ell_1}\) model solved by SAP-ADMM yields
cleaner backgrounds and sharper digit contours than the
other tested models solved by their corresponding algorithms.
The \(\ell_1\)-\(\ell_1\) model solved by
SAP-ADMM\textsuperscript{1} produces the most blurred digits,
consistent with its largest average GMSD value in
Table~\ref{tab:mnist_four_models_compact}.
These results indicate that, even when \(D\) does not have full
row rank, SAP-ADMM remains numerically effective in solving the
\(\ell_1\)-\(C_{\ell_1}\) model and performs well in image denoising.

\section{Conclusions}
\label{sec:conclusions}

In this paper, we studied the structured sparse optimization problem~\eqref{ob0} with a nonsmooth convex loss function and a nonconvex $\ell_0$ penalty on $Dx$. Under suitable assumptions, we established a sequence of equivalent reformulations for problem~\eqref{ob0}, including exact penalty reformulations and exact continuous relaxations, which culminate in a tractable problem~\eqref{rp1}. We proved that $x^*$ is a global minimizer of
problem~\eqref{ob0} if and only if $(x^*,Dx^*)$ is a global
minimizer of problem~\eqref{rp1}. Furthermore, a point $\hat x$ is a local minimizer of problem~\eqref{ob0}, with $D\hat x$ satisfying a desired lower bound property, if and only if $(\hat x, D\hat x)$ is a weak $d$-stationary point of problem~\eqref{rp1}.
Building upon this reformulation framework, we proposed a SAP-ADMM algorithm to solve problem~\eqref{rp1}. We first showed that, after finitely many iterations, SAP-ADMM reduces to either the fast KM iteration or the Halpern iteration. Based on this property, we establish the global convergence of both 
$\{(\bar{x}^k,\bar{y}^k)\}$ and $\{(x^k,y^k)\}$ generated by the ADMM algorithm to the common weak
$d$-stationary point of \eqref{rp1} without relying on the K\L\ property. Under different parameter settings corresponding to the two iteration regimes, we derived either a nonasymptotic convergence rate of \(O(1/k)\) or an asymptotic convergence rate of \(o(1/k)\). Finally, numerical experiments confirmed the effectiveness of both the proposed model and the SAP-ADMM algorithm.

\medskip

\textbf{Code availability.}
The Python code for the numerical experiments in this paper is available
in a public GitHub repository at
\url{https://github.com/TAO-HY/SAP-ADMM-for-Structured-Cardinality-Optimization}.

\medskip

\textbf{Declarations.}
The authors declare that they have no conflicts of interest.


\bibliographystyle{siamplain}
\bibliography{myreferences}
\end{document}

%% file: ex_shared.tex
\usepackage{lipsum}
\usepackage{amsfonts}
\usepackage{graphicx}
\usepackage{epstopdf}
\usepackage{algorithm}
\usepackage[noEnd=false,indLines=true]{algpseudocodex}
\ifpdf
  \DeclareGraphicsExtensions{.eps,.pdf,.png,.jpg}
\else
  \DeclareGraphicsExtensions{.eps}
\fi

\newsiamremark{remark}{Remark}
\newsiamremark{hypothesis}{Hypothesis}
\crefname{hypothesis}{Hypothesis}{Hypotheses}
\newsiamthm{claim}{Claim}
\newsiamremark{fact}{Fact}
\crefname{fact}{Fact}{Facts}

\headers{SAP-ADMM for Structured Cardinality Optimization}{W. Bian, H. Y. Tao, and F.Wu}

\title{A Safeguarded Accelerated Proximal ADMM Algorithm for Solving Structured Cardinality Penalized Optimization Problems\thanks{Submitted to the editors DATE.
\funding{The first author was partially supported by the National Natural Science Foundation of China grants \textcolor{blue}{(12425115)}.}}}

\author{Wei Bian\thanks{Corresponding author. School of Mathematics, Harbin Institute of Technology, Harbin, China
  (\email{bianweilvse520@163.com}).}
\and Hongyuan Tao\thanks{School of Mathematics, Harbin Institute of Technology, Harbin, China 
  (\email{taohy163@163.com}, \email{wufanmath@163.com}).}
\and Fan Wu\footnotemark[3]}

\usepackage{amsopn}


%% file: myreferences.bib
@article{LeCun1998,
    title={Gradient-based learning applied to document recognition},
    author={LeCun, Yann and Bottou, L{\'e}on and Bengio, Yoshua and Haffner, Patrick},
    journal={Proceedings of the IEEE},
    volume={86},
    number={11},
    pages={2278-2324},
    year={1998},
}

@article{Wang2004,
    title={Image quality assessment: from error visibility to structural similarity},
    author={Wang, Zhou and Bovik, Alan C. and Sheikh, Hamid R. and Simoncelli, Eero P.},
    journal={IEEE Transactions on Image Processing},
    volume={13},
    number={4},
    pages={600-612},
    year={2004},
}

@article{Xue2014,
    title={Gradient magnitude similarity deviation: a highly efficient perceptual image quality index},
    author={Xue, Wufeng and Zhang, Lei and Mou, Xuanqin and Bovik, Alan C.},
    journal={IEEE Transactions on Image Processing},
    volume={23},
    number={2},
    pages={684-695},
    year={2014},
}

@article{bot2020The,
	title={The proximal alternating direction method of multipliers in the nonconvex setting: convergence analysis and rates},
	author={Bo\c{t}, Radu Ioan and Nguyen, Dang-Khoa},
	journal={Mathematics of Operations Research},
	volume={45},
	number={2},
	pages={682-712},
	year={2020},
}

@article{Bolte2018NO,
	title={Nonconvex {L}agrangian-based optimization: monitoring schemes and global convergence},
	author={Bolte, J\'{e}r\^{o}me and Sabach, Shoham and Teboulle, Marc},
	journal={Mathematics of Operations Research},
	volume={43},
	number={4},
	pages={1210-1232},
	year={2018},
}

@article{Li2015Glo,
	title={Global convergence of splitting methods for nonconvex composite optimization},
	author={Li, Guoyin and Pong, Ting Kei},
	journal={SIAM Journal on Optimization},
	volume={25},
	number={4},
	pages={2434-2460},
	year={2015},
}

@article{Sun2025Ac,
	title={Accelerating preconditioned {ADMM} via degenerate proximal point mappings},
	author={Sun, Dengfeng and Yuan, Yancheng and Zhang, Guojun and Zhao, Xinyuan},
	journal={SIAM Journal on Optimization},
	volume={35},
	number={2},
	pages={1165-1193},
	year={2025},
}

@article{TRAN2018a,
	title={A smooth primal-dual optimization framework for nonsmooth composite convex minimization},
	author={Tran-Dinh, Quoc and Fercoq, Olivier and Cevher, Volkan},
	journal={SIAM Journal on Optimization},
	volume={28},
	number={1},
	pages={96-134},
	year={2018},
}

@article{bot2015AV,
	title={A variable smoothing algorithm for solving convex optimization problems},
	author={Bo\c{t}, Radu Ioan and Hendrich, Christopher},
	journal={TOP},
	volume={23},
	number={},
	pages={124-150},
	year={2015},
}

@article{bot2020V,
	title={Variable smoothing for convex optimization problems using stochastic gradients},
	author={Bo\c{t}, Radu Ioan and B\"{o}hm, Axel},
	journal={Journal of Scientific Computing},
	volume={85},
	number={2},
	pages={},
	year={2020},
}

@article{Ga1976,
	title={A dual algorithm for the solution of nonlinear variational problems via finite element approximation},
	author={Gabay, Daniel and Mercier, Bertrand},
	journal={Computers \& Mathematics with Applications},
	volume={2},
	number={1},
	pages={17-40},
	year={1976},
}

@article{EC1992o,
	title={On the {D}ouglas-{R}achford splitting method and the proximal point algorithm for maximal monotone operators},
	author={Eckstein, Jonathan and Bertsekas, Dimitri P},
	journal={Mathematical Programming},
	volume={55},
	number={3},
	pages={293-318},
	year={1992},
}

@article{lieder2021convergence,
  title={On the convergence rate of the {H}alpern-iteration},
  author={Lieder, Felix},
  journal={Optimization Letters},
  volume={15},
  number={2},
  pages={405--418},
  year={2021},
  publisher={Springer}
}

@article{xu2002another,
  title={Another control condition in an iterative method for nonexpansive mappings},
  author={Xu, Hong-Kun},
  journal={Bulletin of the Australian Mathematical Society},
  volume={65},
  number={1},
  pages={109--113},
  year={2002},
  publisher={Cambridge University Press}
}

@article{L2019A,
	title={A successive difference-of-convex approximation method for a class of nonconvex nonsmooth optimization problems},
	author={Liu, Tianxiang and Pong, Ting Kei and Takeda, Akiko},
	journal={Mathematical Programming},
	volume={176},
	number={1-2},
	pages={339-367},
	year={2019},
}

@inproceedings{Par2015C,
  title={Convex fused lasso denoising with non-convex regularization and its use for pulse detection},
  author={Parekh, Ankit and Selesnick, Ivan W},
  booktitle={2015 IEEE Signal Processing in Medicine and Biology Symposium (SPMB)},
  pages={1--6},
  year={2015},
  organization={IEEE}
}

@article{Ti2005Sp,
	title={Sparsity and smoothness via the fused lasso},
	author={Tibshirani, Robert and Saunders, Michael and Rosset, Saharon and Zhu, Ji and Knight, Keith},
	journal={Journal of the Royal Statistical Society Series B-Statistical Methodology},
	volume={67},
	number={1},
	pages={91-108},
	year={2005},
}

@article{Ti2011The,
	title={The solution path of the generalized Lasso},
	author={Tibshirani, Ryan J and Taylor, Jonathan},
	journal={The Annals of Statistics},
	volume={39},
	number={3},
	pages={1335-1371},
	year={2011},
}

@article{Wu2025An,
	title={An inexact projected regularized Newton method for fused zero-norms regularization problems},
	author={Wu, Yuqia and Pan, Shaohua and Yang, Xiaoqi},
	journal={Journal of Machine Learning Research},
	volume={25},
	number={373},
	pages={1-48},
	year={2024},
}

@article{Xie2026N,
	title={Nonconvex truncated conditional value at risk-based sparse linear regression},
	author={Xie, Boyi and Wu, Zhongming and Li, Min},
	journal={European Journal of Operational Research},
	volume={328},
	number={1},
	pages={246-257},
	year={2026},
}

@article{Le2015D,
	title={DC approximation approaches for sparse optimization},
	author={Le Thi, H A and Dinh, T Pham and Le, H M and Vo, X T},
	journal={European Journal of Operational Research},
	volume={244},
	number={1},
	pages={26-46},
	year={2015},
}

@article{Char2008R,
	title={Restricted isometry properties and nonconvex compressive sensing},
	author={Chartrand, Rick and Staneva, Valentina},
	journal={Inverse Problems},
	volume={24},
	number={3},
	pages={},
	year={2008},
}

@article{Bian2017O,
	title={Optimality and complexity for constrained optimization problems with nonconvex regularization},
	author={Bian, Wei and Chen, Xiaojun},
	journal={Mathematics of Operations Research},
	volume={42},
	number={4},
	pages={1063-1084},
	year={2017},
}

@article{Pel2008A,
	title={A bilinear formulation for vector sparsity optimization},
	author={Peleg, D and Meir, R},
	journal={Signal Processing},
	volume={88},
	number={5},
	pages={375-389},
	year={2008},
}

@article{Chen2010L,
	title={Lower bound theory of nonzero entries in solutions of $\ell_2-\ell_p$ minimization},
	author={Chen, Xiaojun and Xu, Fengmin and Ye, Yinyu},
	journal={SIAM Journal on Scientific Computing},
	volume={32},
	number={5},
	pages={2832-2852},
	year={2010},
}

@article{Candes2006Robust,
	title={Robust uncertainty principles: exact signal reconstruction from highly incomplete frequency information},
	author={Cand\`{e}s, E and Romberg, J and Tao, T},
	journal={IEEE Transactions on Information Theory},
	volume={52},
	number={2},
	pages={489-509},
	year={2006},
}

@article{Beck2009A,
	title={A fast iterative shrinkage-thresholding algorithm for linear inverse problems},
	author={Beck, A and Teboulle, M},
	journal={SIAM Journal on Imaging Sciences},
	volume={2},
	number={1},
	pages={183-202},
	year={2009},
}

@article{Zhou2021N,
	title={{N}ewton method for $\ell_0$ regularized optimization},
	author={Zhou, Shenlong and Pan, Lili and Xiu, Naihua},
	journal={Numerical Algorithms},
	volume={88},
	number={4},
	pages={1541-1570},
	year={2021},
}

@article{Ye2026S,
	title={Subspace {N}ewton's method for $\ell_0$ regularized optimization problems with box constraints},
	author={Ye, Yuge and Li, Qingna},
	journal={Journal of Scientific Computing},
	volume={106},
	number={68},
	pages={},
	year={2026},
}

@article{Bian2023A,
	title={Accelerated smoothing hard thresholding algorithms for $\ell_0$ regularized nonsmooth convex regression problem},
	author={Bian, Wei and Wu, Fan},
	journal={Journal of Scientific Computing},
	volume={96},
	number={33},
	pages={},
	year={2023},
}

@article{Lu2012Sparse,
	title={Sparse approximation via penalty decomposition methods},
	author={Lu, Z and Zhang, Y},
	journal={SIAM Journal on Optimization},
	volume={23},
	number={4},
	pages={2448-2478},
	year={2013},
}

@article{Lu2014Iterative,
	title={Iterative hard thresholding methods for $\ell_0$ regularized convex cone programming},
	author={Lu, Z},
	journal={Mathematical Programming},
	volume={147},
	number={1-2},
	pages={125-154},
	year={2014},
}

@article{bot2023fast,
  title   = {Fast {Krasnosel'ski\u{\i}--Mann} Algorithm with a
             Convergence Rate of the Fixed Point Iteration of {$o(1/k)$}},
  author  = {Bo{\c{t}}, Radu Ioan and Nguyen, Dang-Khoa},
  journal = {SIAM Journal on Numerical Analysis},
  volume  = {61},
  number  = {6},
  pages   = {2813-2843},
  year    = {2023},
}

@article{Blumensath2008Iterative,
	title={Iterative thresholding for sparse approximations},
	author={Blumensath, T and Davies, M},
	journal={Journal of Fourier Analysis and Applications},
	volume={14},
	number={5-6},
	pages={629-654},
	year={2008},
}

@article{Blumensath2009Iterative,
	title={Iterative hard thresholding for compressed sensing},
	author={Blumensath, T and Davies, M},
	journal={Applied and Computational Harmonic Analysis},
	volume={27},
	number={3},
	pages={265-274},
	year={2009},
}

@article{F2001V,
	title={Variable selection via nonconcave penalized likelihood and its oracle properties},
	author={Fan, Janqing and Li, Runze},
	journal={Journal of the American Statistical Association},
	volume={96},
	number={456},
	pages={1348-1360},
	year={2001},
}

@article{E2005Q,
	title={Quantile smoothing of array CGH data},
	author={Eilers, Paul H C and Menezes, Renee X de},
	journal={Bioinformatics},
	volume={21},
	number={7},
	pages={1146-1153},
	year={2005},
}

@article{J2020F,
	title={Fast nonconvex deconvolution of calcium imaging data},
	author={Jewell, Sean W and Hocking, Toby Dylan and Fearnhead, Paul and Witten, Daniela M},
	journal={Biostatistics},
	volume={21},
	number={4},
	pages={709-726},
	year={2020},
}

@article{L1997v,
	title={Variable fusion: a new adaptive signal regression method},
	author={Land, Stephanie R and Friedman, Jerome H},
	journal={Technical report. Technical Report 656, Department of Statistics, Carnegie Mellon University},
	volume={},
	number={},
	pages={},
	year={1997},
}

@book{rockafellar2009variational,
	author={Rockafellar, R. Tyrrell and Wets, Roger J. B.},
	title={Variational Analysis},
	publisher={Springer},
	address={Berlin},
	year={2009},
    edition={3rd},
	
}

@book{RocC,
	 author    = {Rockafellar, R. Tyrrell},
  title     = {Convex Analysis},
  publisher = {Princeton University Press},
  address   = {Princeton, NJ},
  year      = {1970}
}

@article{C2017,
	title={Computing {B}-stationary points of nonsmooth {DC} programs},
	author={Pang, Jong Shi and Razaviyayn, Meisam and Alvarado, Alberth},
	journal={Mathematics of Operations Research},
	volume={42},
	number={1},
	pages={95-118},
	year={2017},
}

@article{Bian2020a,
  title={A smoothing proximal gradient algorithm for nonsmooth convex regression with cardinality penalty},
  author={Bian, W and Chen, X},
  journal={SIAM Journal on Numerical Analysis},
  volume={58},
  number={1},
  pages={858-883},
  year={2020},
}
